\documentclass[reqno,11pt]{amsart}

\usepackage{latexsym}
\usepackage[margin=3cm]{geometry}
\usepackage{amscd}
\usepackage{amsthm}
\usepackage{marginnote}
\usepackage{subfigure,amsmath,amsfonts,amssymb,epsfig}
\usepackage{graphics}
\usepackage{xy}
\usepackage{mathrsfs}
\usepackage{tipa}

\usepackage[latin1]{inputenc}

\usepackage{subfigure,amsmath}

\input xy
\xyoption{all}

\def\!{\mskip-\thinmuskip}
\def\,{\mskip\thinmuskip}
\def\;{\mskip\thickmuskip}

\newtheorem{theorem}{Theorem}[section]
\newtheorem{corollary}[theorem]{Corollary}
\newtheorem{proposition}[theorem]{Proposition}

\theoremstyle{remark}
\newtheorem{remark}[theorem]{Remark}
\newtheorem{example}[theorem]{Example}

\numberwithin{equation}{section}
\allowdisplaybreaks

\author[Michael J.\ Schlosser]{Michael J.\ Schlosser$^*$}
\address{Fakult\"at f\"ur Mathematik, Universit\"at Wien,
Oskar-Morgenstern-Platz~1, A-1090 Vienna, Austria}
\email{michael.schlosser@univie.ac.at}
\urladdr{http://www.mat.univie.ac.at/{\textasciitilde}schlosse}
\thanks{$^*$Partly supported by FWF Austrian Science Fund
grant F50-08 within the SFB
``Algorithmic and enumerative combinatorics''.}

\author[Meesue Yoo]{Meesue Yoo$^{**}$}
\address{Fakult\"at f\"ur Mathematik, Universit\"at Wien,
Oskar-Morgenstern-Platz~1, A-1090 Vienna, Austria}
\email{meesue.yoo@univie.ac.at}
\thanks{$^{**}$Fully supported by FWF Austrian Science Fund
grant F50-08 within the SFB
``Algorithmic and enumerative combinatorics''.}

\title[Elliptic extensions of rook models]{Elliptic extensions
of the alpha-parameter model and the rook model for matchings}

\subjclass[2010]{Primary 05A19;
Secondary 05A15, 05A30, 11B65, 11B73, 11B83}

\keywords{rook numbers, $q$-analogues, elliptic extensions,
alpha-parameter model, matchings}

\newcommand{\ta}{\theta}
\newcommand{\C}{\mathbb C}

\newcommand{\N}{\mathbb N}

\begin{document}

\begin{abstract}
We construct elliptic extensions of the alpha-parameter rook model 
introduced by Goldman and Haglund and of the rook model for
matchings of Haglund and Remmel. In particular, we extend the
product formulas of these models to the elliptic setting.
By specializing the parameter $\alpha$ in our elliptic extension
of the alpha-parameter model and the shape of the Ferrers board
in different ways, we obtain elliptic analogues of the Stirling numbers
of the first kind and of the Abel polynomials, and obtain an
$a,q$-analogue of the matching numbers.
We further generalize the rook theory model for matchings by introducing
$\mathbf l$-lazy graphs which correspond to $\mathbf l$-shifted boards, 
where $\mathbf l$ is a finite vector of positive integers.
The corresponding elliptic product formula generalizes Haglund and Remmel's
product formula for matchings already in the non-elliptic basic case.
\end{abstract}

\maketitle


\section{Introduction}\label{sec:intro}


Since the introduction of rook theory by Kaplansky and Riordan~\cite{KR},
the theory has thrived and developed further by revealing connections 
to, for instance, orthogonal polynomials~\cite{Ge, GJW4}, hypergeometric 
series \cite{H0}, $q$-analogues and permutation statistics~\cite{D, GR},
algebraic geometry~\cite{Ding1, Ding2}, and many more.
Within rook theory itself, various models have been introduced, 
including a $p,q$-analogue of rook numbers~\cite{BR, RW, WaW}, the
$j$-attacking model~\cite{RW}, the matching model~\cite{HR},
the augmented rook model~\cite{McR}
which includes all other models as special cases, etc.
In previous work~\cite{SY0,SY1}, 
the authors have constructed elliptic extensions of the aforementioned 
rook theory models and obtained corresponding product formulas.

In this work, we construct elliptic extensions of two rook theory models:
the alpha parameter rook model introduced by Goldman and
Haglund~\cite{GH} and the
matching model of Haglund and Remmel~\cite{HR}. The alpha-parameter 
model, as explained in \cite{GH}, is a slight generalization of the
$i$-creation model. It connects to several other combinatorial models,
including polynomial sequences of binomial type, permutations of
multisets, Abel polynomials, Bessel polynomials and matchings, and so on. 
Our elliptic extension lays the foundations for raising those
connections to the elliptic level. 

In our construction of an elliptic analogue of the matching model,
we actually consider a model that generalizes the original model
of Haglund and Remmel already in the non-elliptic, basic case.
In particular, we consider matchings on specific graphs which
we call ``$\mathbf l$-lazy graphs'' with respect
to an $N$-dimensional vector of positive integers,
$\mathbf l=(l_1,l_2,\dots, l_N)$. The original matching model can
be realized from the generalized model by setting
$N=2n-1$ and $\mathbf l=(1,1,\dots, 1)$.
For the new model, we are able to prove a product formula
involving elliptic rook numbers for matchings on $\mathbf l$-lazy graphs,
a result which generalizes the corresponding product formula of
Haglund and Remmel~\cite{HR}.

In Section~\ref{sec:ew} we define elliptic weights and review some of the
elementary identities useful for dealing with them. Section~\ref{sec:alpha}
is devoted to the elliptic extension of the alpha-parameter model,
together with some applications. Finally, Section~\ref{sec:match}
features an elliptic extension of the rook theory of matchings.


\section{Elliptic weights}\label{sec:ew}


In this section, we define the elliptic weights which we utilize 
to weight cells in Ferrers boards. (For the definition of Ferrers boards,
see Section \ref{sec:alpha}). We start by explaining what elliptic 
functions are.

A complex function is called elliptic, if it is a doubly-periodic,
meromorphic function on $\C$. It is well-known that such functions
can be expressed in terms of ratios of theta functions (cf.\ \cite{W}).
We will use the following (multiplicative) notation for theta functions.
First, we define the \emph{modified Jacobi theta function} with argument $x$
and nome $p$ by
$$\theta(x;p)= \prod_{j\ge 0}((1-p^j x)(1- p^{j+1}/x)),\qquad
\theta(x_1,\dots, x_m;p)=\prod_{k=1}^m \theta(x_k;p),$$
where $x,x_1,\dots, x_m\ne 0$, $|p|<1$. 
Further, we define the {\em theta shifted factorial}
(or {\em $q,p$-shifted factorial}) by
\begin{equation*}
(a;q,p)_n = \begin{cases}
\prod^{n-1}_{k=0} \theta (aq^k;p),&\quad n = 1, 2, \ldots\,,\cr
1,&\quad n = 0,\cr
1/\prod^{-n-1}_{k=0} \theta (aq^{n+k};p),&\quad n = -1, -2, \ldots.
\end{cases}
\end{equation*}
We frequently write
\begin{equation*}
(a_1, a_2, \ldots, a_m;q, p)_n = \prod^m_{k=1} (a_k;q,p)_n,
\end{equation*}
for brevity.
Notice that for $p=0$ we have  $\theta (x;0) = 1-x$ and,
hence, $(a;q, 0)_n = (a;q)_n$ is just the usual {\em $q$-shifted factorial}
in base $q$ (cf.\ \cite{GRhyp}).
The parameters $q$ and $p$
in $(a;q,p)_n$ are called the {\em base} and {\em nome}, respectively. 
In analogy to the theories of ordinary and basic hypergeometric series
(cf.\ \cite{GRhyp}) there exists also a (rather young) theory of
hypergeometric series involving theta shifted factorials, namely
of theta, modular, and elliptic hypergeometric series,
see \cite[Chapter~11]{GRhyp}.

The modified Jacobi theta functions satisfy the following
basic properties which are essential in the theory of elliptic
hypergeometric series:
\begin{equation*}\label{tif}
\ta(x;p)=-x\,\ta(1/x;p),
\end{equation*}
\begin{equation}\label{p1id}
\ta(px;p)=-\frac 1x\,\ta(x;p),
\end{equation}
and the {\em addition formula}
\begin{equation}\label{addf}
\ta(xy,x/y,uv,u/v;p)-\ta(xv,x/v,uy,u/y;p)
=\frac uy\,\ta(yv,y/v,xu,x/u;p)
\end{equation}
(cf.\ \cite[p.~451, Example 5]{WhW}).

The three-term relation in \eqref{addf},
containing four variables
and four factors of theta functions in each term, is the ``smallest''
addition formula connecting products of theta functions with general
arguments.
Note that in the theta function $\theta(x;p)$ one cannot let $x\to 0$
(unless one first lets $p\to 0$) since $x$ is an essential singularity.
(For this reason elliptic analogues of
$q$-series identities usually contain many parameters.)

The elliptic identities we shall consider all involve terms
which are elliptic (with the same periods) in all of its parameters.
Spiridonov~\cite{Sp} refers to such multivariate functions as
{\em totally elliptic}; they are by nature
{\em well-poised} and {\em balanced} (see also \cite[Chapter~11]{GRhyp}).

Following the setup used in our earlier paper on elliptic rook
numbers~\cite{SY0}, we define the {\em elliptic weights}
$w_{a,b;q,p}(k)$ and $W_{a,b;q,p}(k)$,
depending on two independent parameters $a$ and $b$, base $q$,
nome $p$, and integer parameter $k$ by
\begin{subequations}
\begin{align}
w_{a,b;q,p}(k)&=\frac{\theta(aq^{2k+1},bq^{k},aq^{k-2}/b;p)}
{\theta(aq^{2k-1},bq^{k+2},aq^k /b;p)}q,\label{def:smallelpwt}\\\intertext{and}
W_{a,b;q,p}(k)&=
\frac{\theta(aq^{1+2k},bq,bq^2,aq^{-1}/b,a/b;p)}
{\theta(aq,bq^{k+1},bq^{k+2},aq^{k-1}/b, aq^k /b;p)}q^k,\label{def:bigelpwt}
\end{align}
\end{subequations}
respectively. It is clear that if $k$ is a positive integer, Equations
\eqref{def:smallelpwt} and \eqref{def:bigelpwt} imply that
\begin{equation}\label{eqn:wrelW}
W_{a,b;q,p}(k)=\prod_{j=1}^k w_{a,b;q,p}(j).
\end{equation}
We refer to the $w_{a,b;q,p}(k)$ as {\em small} weights (these will
correspond to the weights of single squares in the Ferrers boards)
and to the $W_{a,b;q,p}(k)$ as {\em big} weights (these will
correspond to partial columns of a Ferrers board). 
Note that the weights $w_{a,b;q,p}(k)$ and $W_{a,b;q,p}(k)$ also can be
defined for arbitrary (complex) value $k$ which is clear from the definition.

We wil make frequent use of the following two properties:
\begin{subequations}\label{eqn:shifts}
\begin{align}
w_{a,b;q,p}(k+n)&=w_{aq^{2k},bq^k;q,p}(n),\label{wshift}\\
W_{a,b;q,p}(k+n)&=W_{a,b;q,p}(k)\,
W_{aq^{2k},bq^k;q,p}(n).\label{Wshift}
\end{align}
\end{subequations}

\begin{remark}
\begin{enumerate}
 \item  The small weight $w_{a,b;q,p}(k)$ (and so the big one) is indeed
elliptic in its parameters (i.e., totally elliptic).
If we write $q=e^{2\pi i\sigma}$,
$p=e^{2\pi i\tau}$, $a=q^\alpha$ and $b=q^\beta$ with complex $\sigma$,
$\tau$, $\alpha$, $\beta$ and $k$, then the small weight
$w_{a,b;q,p}(k)$ is clearly periodic in $\alpha$ with period $\sigma^{-1}$.
A simple computation involving \eqref{p1id} further shows that
$w_{a,b;q,p}(k)$ is also periodic in $\alpha$ with period $\tau\sigma^{-1}$.
The same applies to $w_{a,b;q,p}(k)$ as a function in $\beta$ (or $k$)
with the same two periods $\sigma^{-1}$ and $\tau\sigma^{-1}$.
\item For $p\to 0$, the small and big weights reduce to
\begin{subequations}\label{def:qwt}
\begin{align}
w_{a,b;q}(k)&=\frac{(1-aq^{2k+1})(1-bq^{k})(1-aq^{k-2}/b)}
{(1-aq^{2k-1})(1-bq^{k+2})(1-aq^k /b)}q,~\text{ and }\label{def:smallqwt}\\
W_{a,b;q}(k)&=
\frac{(1-aq^{1+2k})(1-bq)(1-bq^2)(1-aq^{-1}/b)(1-a/b)}
{(1-aq)(1-bq^{k+1})(1-bq^{k+2})(1-aq^{k-1}/b)(1-aq^k /b)}q^k,\label{def:bigqwt}
\end{align}
\end{subequations}
respectively.
In the $a,b;q$-weights in \eqref{def:qwt},
we may let $b\to 0$ (or $b\to\infty$) to obtain ``$a,0;q$-weights'',
or in short, ``$a;q$-weights'':
\begin{equation*}\label{def:aqwt}
w_{a;q}(k)=\frac{(1-aq^{2k+1})}{(1-aq^{2k-1})}q^{-1},\qquad\text{and}\qquad
W_{a;q}(k)=\frac{(1-aq^{1+2k})}{(1-aq)}q^{-k}.
\end{equation*}
Note that by writing $q=e^{ix}$ and $a=e^{i(2c+1)x}$, $c\in\mathbb N$,
the $a;q$-weights
can be written as quotients of Chebyshev polynomials
of the second kind.

Also, in \eqref{def:qwt}, we may let $a\to 0$ (or $a\to\infty$) to obtain
``$b;q$-weights''.
Importantly, if in \eqref{def:qwt} we first let $b\to 0$ and then
$a\to\infty$ (or, equivalently, first let $a\to 0$ and then $b\to 0$),
we obtain the familiar $q$-weights
\begin{equation*}
w_q(k)=q\qquad\text{and}\qquad W_q(k)=q^k,
\end{equation*}
respectively.
\end{enumerate}
\end{remark}

We also define an elliptic analogue of the
$q$-number $[z]_q=\frac{1-q^z}{1-q}$ by
\begin{equation*}\label{elln}
[z]_{a,b;q,p}=\frac{\theta(q^z, aq^z, bq^2, a/b;p)}
{\theta(q,aq,bq^{z+1},aq^{z-1}/b;p)}.
\end{equation*}
Using the addition formula for theta functions~\eqref{addf}, it is
straightforward to verify that the thus defined elliptic numbers satisfy
\begin{equation}\label{recelln}
[z]_{a,b;q,p} = [z-1]_{a,b;q,p}+W_{a,b;q,p}(z-1).
\end{equation}
In case $z=n$ is a nonnegative integer, \eqref{recelln} constitutes a
recursion which, together with $W_{a,b;q,p}(0)=1$,
uniquely defines any elliptic number
$[n]_{a,b;q,p}$, namely,
$$[n]_{a,b;q,p}= 1+W_{a,b;q,p}(1)+\cdots +W_{a,b;q,p}(n-1).$$
More generally, by \eqref{addf} we have the following useful identity
\begin{equation}\label{recellny}
[z]_{a,b;q,p} = [y]_{a,b;q,p}+W_{a,b;q,p}(y)[z-y]_{aq^{2y},bq^y;q,p}
\end{equation}
which reduces to \eqref{recelln} for $y=z-1$.

\begin{remark}
In \cite{Schl1}, the first author, 
in analogy to the $q$-binomial coefficients 
\begin{equation*}
\begin{bmatrix}n\\k\end{bmatrix}_q
:=\frac{(q^{1+k};q)_{n-k}}
{(q;q)_{n-k}}=\frac{[n]_q!}{[k]_q![n-k]_q!},
\end{equation*}
where $[0]_q!=1$ and $[n]_q!=[n]_q[n-1]_q!$,
defined the elliptic binomial coefficients 
\begin{equation}\label{ellbin}
\begin{bmatrix}n\\k\end{bmatrix}_{a,b;q,p}:=
\frac{(q^{1+k},aq^{1+k},bq^{1+k},aq^{1-k}/b;q,p)_{n-k}}
{(q,aq,bq^{1+2k},aq/b;q,p)_{n-k}}.
\end{equation}
These satisfy a nice recursion 
\begin{align}
&\begin{bmatrix}n+1\\k\end{bmatrix}_{a,b;q,p}=
\begin{bmatrix}n\\k\end{bmatrix}_{a,b;q,p}
+\begin{bmatrix}n\\k-1\end{bmatrix}_{a,b;q,p}
\,W_{aq^{k-1},bq^{2k-2};q,p}(n+1-k)
\qquad \text{for $n,k\in\N_0$},\notag\\
\intertext{with the initial conditions}
&\begin{bmatrix}0\\0\end{bmatrix}_{a,b;q,p}=1,\qquad
\begin{bmatrix}n\\k\end{bmatrix}_{a,b;q,p}=0
\qquad\text{for\/ $n\in\N_0$, and\/
$k\in-\N$ or $k>n$},\label{recw}
\end{align}
where $\N$ and $\N_0$ denote the sets of positive and of nonnegative
integers, respectively.

Combinatorially, the elliptic binomial coefficient
in \eqref{ellbin}
can be interpreted in terms of weighted lattice paths in $\mathbb Z^2$
see \cite{Schl0}.
More precisely, \eqref{ellbin} is the
weighted-area generating function for
paths starting in $(0,0)$ and ending in $(k,n-k)$ composed of unit steps going
north or east only, when the weight of each cell
(with $(s,t)$ as the north-east corner) ``covered'' by the path
is defined to be $w_{aq^{s-1},bq^{2s-2};q,p}(t)$.
Then it can be shown that the sum of weighted areas below the paths 
satisfies the same recursion \eqref{recw} by distinguishing the 
last step of the path whether it is vertical or horizontal. Now,
the elliptic number $[n]_{a,b;q,p}$ is just a short-hand notation for
\begin{equation*}
[n]_{a,b;q,p}=\begin{bmatrix}n\\1\end{bmatrix}_{a,b;q,p},
\end{equation*}
the weighted enumeration of all paths starting in 
$(0,0)$ and ending in $(1,n-1)$.

If we take the limit $p\to 0$, $a\to 0$, and $b\to 0$ (in this order),
then we recover the usual $q$-binomial coefficients, i.e., 
$$\lim_{b\to 0}(\lim_{a\to 0}(\lim_{p\to 0}[z]_{a,b;q,p}))=[z]_q.$$
\end{remark}


\section{Elliptic extension of the alpha-parameter
model}\label{sec:alpha}


In \cite[Section~6]{GH}, Goldman and Haglund
introduced the \emph{alpha-parameter model} which generalizes
the original rook theory \cite{KR} as well as 
the $i$-creation model that they
introduced in the same paper \cite{GH}. In \cite[Section~7]{GH}, 
they also derived $q$-analogues of their results. The purpose of this section
is to extend the alpha-parameter model to the elliptic setting. 
Before we give an explicit description in the elliptic setting,
we review the classical case. 

We consider a \emph{board} to be a finite subset of the
$\mathbb N\times\mathbb N$ grid, and label the columns from left to
right with $1,2,3,\dots$, and the rows from bottom to top with
$1,2,3,\dots$. We use the notation $(i,j)$ to denote the cell in the
$i$-th column from the left and the $j$-th row from the bottom.
For technical reasons, in our proofs, we sometimes find it convenient to
extend the $\mathbb N\times \mathbb N$ grid to $\mathbb N\times \mathbb Z$
where the row below the row $1$ is labeled by $0$ and the row labels 
decrease by $1$ as they go down.

Let $B(b_1, \dots, b_n)$ denote the set of cells 
\begin{displaymath}
B=B(b_1,\dots, b_n)=\{(i,j)~|~ 1\le i\le n,~ 1 \le j\le b_i\}.
\end{displaymath}
Note that $b_i$'s are allowed to be zero.
If a board $B$ can be represented by the set $B(b_1, \dots, b_n)$ for 
some nonnegative integer $b_i$'s, then the board $B$ is called 
a \emph{skyline board}. If in addition those $b_i$'s are nondecreasing, 
that is, $0\le b_1\le \cdots \le b_n$, 
then the board $B=B(b_1,\dots, b_n)$ is called a \emph{Ferrers board}.

Classical rook theory studies the number of way of choosing $k$ cells 
in the given board $B$, denoted by $r_k(B)$, so that no two have
a common coordinate, that is, 
no two cells lie in the same row or in the same column. 
For this, we say that we place $k$-nonattacking rooks on the board.

The \emph{$i$-creation rook placement} can be obtained by the 
following process: first choose the columns to place rooks. Then as 
nonattacking rooks are placed in columns, from left to right,
to the right of each rook, $i$ new rows are created until the end of the row,
strictly above where the rook is placed. Given a Ferrers board $B$,
the \emph{$i$-rook number}, $r_k ^{(i)}(B)$, counts the number of 
$i$-creation rook placements of $k$ rooks on $B$, with $r_0 ^{(i)}(B)=1$. 
Figure \ref{fig:i=1} shows an $(i=1)$-creation rook placement in which 
rooks are denoted by $X$'s.

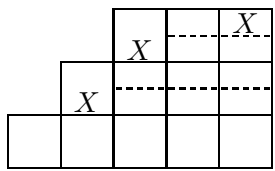
\begin{figure}[ht]
$$
\begin{picture}(100,60)(0,0)
\multiput(0,0)(0,20){2}{\line(1,0){100}}
\multiput(20,40)(0,20){1}{\line(1,0){80}}
\multiput(40,60)(0,20){1}{\line(1,0){60}}
\multiput(0,0)(20,0){1}{\line(0,1){20}}
\multiput(20,0)(20,0){1}{\line(0,1){40}}
\multiput(40,0)(20,0){4}{\line(0,1){60}}
\multiput(41,30)(4,0){15}{\line(1,0){2}}
\multiput(61,50)(4,0){10}{\line(1,0){2}}
\put(25, 21){$X$}
\put(45, 41){$X$}
\put(85, 51){$X$}
\end{picture}$$
\caption{An $i$-creation rook placement for $i=1$.}\label{fig:i=1}
\end{figure}

The $i$-rook number $r_k ^{(i)}(B)$ can be generalized by introducing 
weights to each row of the board. Now we allow to place more than one 
rook in the respective rows, keeping the condition that each column
contains at most one rook. Such placements are called
\emph{file placements} and we use  $\mathcal F_k(B)$ to denote the
set of all $k$-rook file placements in $B$.
Rooks do not create $i$ rows to the right, but we instead weight 
the rows: if there are $u$ rooks in a given row, then the
weight of the row is 
$$
\begin{cases}
1 & \text{ if } ~0\le u\le 1,\\
\alpha(2 \alpha-1)(3\alpha -2)\cdots ((u-1)\alpha -(u-2)),& 
\text{ if } ~u\ge 2.
\end{cases}
$$
The weight of a placement is defined to be the product of the weights
of all the rows and $r_k ^{(\alpha)}(B)$ is defined by the sum of the weights
of all placements $P\in \mathcal F_k(B)$. If $\alpha=0$, then $r_k ^{(0)}(B)$
reduces to the original rook number, and if $\alpha$ is a positive
integer $i$, then $r_k ^{(i)}(B)$ is the $i$-rook number of the
$i$-creation model. The numbers $r_k ^{(\alpha)}(B)$
satisfy the following product
formula, or $\alpha$-factorization theorem (see \cite{GH}):
\begin{equation*}
\prod_{j=1}^n (z+b_j +(j-1)(i-1))=\sum_{k=0}^n r_{n-k} ^{(\alpha)}(B)
\prod_{i=1}^k (z+(i-1)(\alpha-1)).
\end{equation*}

We establish an elliptic analogue of the alpha-parameter model by
assigning elliptic weights to the cells in the board $B$. 
This elliptic analogue contains the $q$-analogue of $r_k ^{(\alpha)}(B)$
that Goldman and Haglund constructed in \cite[Section 7]{GH} as a limit case. 

Let $B$ be a Ferrers board and $P\in \mathcal F_k(B)$ a file placement in $B$.
For each  cell $c\in B$ in $(i,j)$, we define the weight of $c$ to be 
\begin{multline}\label{def:cellwt}
wt_\alpha (c)=\\
\begin{cases}
 1, \hfill\text{ if there is a rook above and in the same column as $c$,}\\
 [(\alpha -1)v(c)+1]_{aq^{2(-j+(\alpha-1)(1-i+r_c(P)))},bq^{-j+(\alpha-1)(1-i+r_c(P))};q,p},
\text{ if $c$ contains a rook,}\\
 W_{aq^{2(-j+(\alpha-1)(1-i+r_c(P)))},bq^{-j+(\alpha-1)(1-i+r_c(P))};q,p}
((\alpha -1)v(c) +1), \hfill\text{ otherwise,}
\end{cases}
\end{multline}
where $v(c)$ is the number of rooks strictly to the left of, and 
in the same row as $c$, and $r_c (P)$ is the number of rooks 
in the north-west region of $c$. 
The weight of the rook placement $P$ is defined to be the product of 
the weights of all cells:
$$wt_\alpha (P)=\prod_{c\in B}wt_\alpha (c).$$
We define an elliptic analogue of $r_k ^{(\alpha)}(B)$ by setting 
\begin{equation*}
r_k ^{(\alpha)}(a,b;q,p;B)=\sum_{P\in\mathcal F_k (B) }wt_\alpha (P).
\end{equation*}
This $r_k ^{(\alpha)}(a,b;q,p;B)$ satisfies the following recursion which 
can be proved by considering the cases whether there is a rook or not 
in the last column of the board.

\begin{proposition}\label{prop:elptrecur}
Let $B$ be a Ferrers board with $l$ columns of height at most $m$,
and $B\cup m$ denote the board obtained by adding the $(l+1)$-st
column of height $m$ to $B$. Then, for any integer $k$, we have 
\begin{multline}\label{eqn:elptrecur}
r_{k+1} ^{(\alpha)}(a,b;q,p;B\cup m)=
[m+(\alpha -1)k]_{aq^{-2(m+(\alpha-1)l)},bq^{-m-(\alpha-1)l};q,p}\,
r_k ^{(\alpha)}(a,b;q,p;B)\\
+W_{aq^{-2(m+(\alpha-1)l)},bq^{-m-(\alpha-1)l};q,p}(m+(\alpha -1)(k+1))\,
r_{k+1} ^{(\alpha)}(a,b;q,p;B),
\end{multline}
assuming that 
\begin{align*}
r_k ^{(\alpha)}(a,b;q,p;B)={}&0\qquad\text{for $k<0$ or $k>l$, and}\\
r_0 ^{(\alpha)}(a,b;q,p;B)={}&1\qquad\text{for $l=0$, i.e.\ for $B$
being the empty board.}
\end{align*}
\end{proposition}

\begin{proof}
Let us compute the coefficient of $r_{k+1} ^{(\alpha)}(a,b;q,p;B)$ which
corresponds to the case when there is no rook in the last column.
To see how the computation goes, we first assume that there is only
one row in $B$ containing all $k+1$ rooks in the row and let the row
coordinate be $j$. Then the product of the weight of the cells
in the last column would be,
from the top, 
\begin{multline*}
W_{aq^{-2(m+(\alpha -1)l)},bq^{-(m+(\alpha-1)l)};q,p}(1) \cdot
W_{aq^{-2(m-1+(\alpha -1)l)},bq^{-(m-1+(\alpha-1)l)};q,p}(1) \cdots\\
\times W_{aq^{-2(j+1+(\alpha -1)l)},bq^{-(j+1+(\alpha-1)l)};q,p}(1)
\cdot W_{aq^{-2(j+(\alpha -1)l)},bq^{-(j+(\alpha-1)l)};q,p}((\alpha-1)(k+1)+1)\\
\times W_{aq^{-2(j-1+(\alpha -1)(l-k-1))},bq^{-(j-1+(\alpha-1)(l-k-1))};q,p}(1)
\cdot W_{aq^{-2(j-2+(\alpha -1)(l-k-1))},bq^{-(j-2+(\alpha-1)(l-k-1))};q,p}(1)\\
\times \cdots \times
W_{aq^{-2(1+(\alpha -1)(l-k-1))},bq^{-(1+(\alpha-1)(l-k-1))};q,p}(1).\qquad\qquad
\end{multline*}
By applying the identities \eqref{eqn:wrelW} and \eqref{eqn:shifts}, 
it is not very hard to see that the above product eventually reduces to 
$$W_{aq^{-2(m+(\alpha -1)l)},bq^{-(m+(\alpha-1)l)};q,p}(m+(\alpha-1)(k+1)).$$
Even if there are several rows containing rooks, say the $j_i$th row
contains $v_{j_i}$ rooks, for $1\le i\le k+1$ and $\sum_{i=1}^k v_{j_i}=k+1$,
by the way we defined the weights of the cells, the product of the
weight of the cells in the last column eventually reduces to 
$$W_{aq^{-2(m+(\alpha -1)l)},bq^{-(m+(\alpha-1)l)};q,p}(m+(\alpha-1)(k+1)),$$
which is the coefficient of $r_{k+1} ^{(\alpha)}(a,b;q,p;B)$ in
\eqref{eqn:elptrecur}. 

Now we consider the case when there is a rook in the last column and
sum up the weights coming from all the possible rook placements.
Let us assume that the rows $j_1,j_2,\dots, j_k$,
from the top, contain $v_{j_i}$ rooks respectively, for $i=1,\dots, k$,
and $\sum_{i=1}^k v_{j_i}=k$. 
If we place the rook in the top cell of the last column, 
then the top cell containing the rook has the weight
$[1]_{aq^{2(-m+(\alpha-1)(-l))},bq^{-m+(\alpha-1)(-l)};q,p}$,
which is just $1$, and all the cells below, being below a rook,
have weight $1$.
If we place the rook in the second top cell, then the top cell itself
has the weight 
$W_{aq^{-2(m+(\alpha -1)l)},bq^{-(m+(\alpha-1)l)};q,p}(1)$ and the other cells have 
weight $1$. If we place the rook in the third top row, then the weight
of the placement would be 
\begin{align*}
 &W_{aq^{-2(m+(\alpha -1)l)},bq^{-(m+(\alpha-1)l)};q,p}(1)\cdot
W_{aq^{-2(m-1+(\alpha -1)l)},bq^{-(m-1+(\alpha-1)l)};q,p}(1)\\
 &=W_{aq^{-2(m+(\alpha -1)l)},bq^{-(m+(\alpha-1)l)};q,p}(2),
\end{align*}
by application of \eqref{Wshift}. Continuing in this way, we can see that
the weights grow as we place rooks in lower rows of the last column.
If we place the rook in the row $j_1$, containing $v_{j_1}$ rooks to
the left, the weight of the placement is 
$$
 W_{aq^{-2(m+(\alpha -1)l)},bq^{-(m+(\alpha-1)l)};q,p}(m-j_1)
 [(\alpha -1)v_{j_1}+1]_{aq^{-2(j_1+(\alpha -1)l)},bq^{-(j_1+(\alpha-1)l)};q,p}.
$$
This is the result after simplifying the product of the weights coming
from the top $m-j_1$ cells using \eqref{Wshift}.
The sum of the rook placements so far is 
\begin{align*}
&[m-j_1]_{aq^{-2(m+(\alpha -1)l)},bq^{-(m+(\alpha-1)l)};q,p}+ \\
&~ W_{aq^{-2(m+(\alpha -1)l)},bq^{-(m+(\alpha-1)l)};q,p}(m-j_1)
[(\alpha -1)v_{j_1}+1]_{aq^{-2(j_1+(\alpha -1)l)},bq^{-(j_1+(\alpha-1)l)};q,p}\\
&= [m-j_1+(\alpha -1)v_{j_1}+1]_{aq^{-2(m+(\alpha -1)l)},bq^{-(m+(\alpha-1)l)};q,p},
\end{align*}
by applying \eqref{recellny}.
The placement of the last rook in row $j_1 -1$ has the weight 
\begin{align*}
&  \underbrace{W_{aq^{-2(m+(\alpha -1)l)},bq^{-(m+(\alpha-1)l)};q,p}
(m-j_1)}_{\text{product of weights of top $m-j_1$ cells}}
\cdot \underbrace{W_{aq^{-2(j_1+(\alpha -1)l)},bq^{-(j_1+(\alpha-1)l)};q,p}
((\alpha-1)v_{j_1}+1)}_{\text{weight of the cell in row $j_1$}}\\
&= ~W_{aq^{-2(m+(\alpha -1)l)},bq^{-(m+(\alpha-1)l)};q,p}(m-j_1 +(\alpha -1)v_{j_1}+1),
\end{align*}
by \eqref{eqn:shifts}, and hence, the sum of the weights of the
rook placements so far is 
$$[m-j_1+(\alpha -1)v_{j_1}+2]_{aq^{-2(m+(\alpha -1)l)},bq^{-(m+(\alpha-1)l)};q,p}.$$
We can see that placing a rook to the right of $v_{j_i}$ rooks contribute 
$(\alpha -1)v_{j_i}$ discrepancy and the elliptic number continues to increase 
as we place the rook in lower rows. We continue in this way. If we place
the rook in the bottom most cell in the last column,
the weight of the placement is 
\begin{align*}
 & \quad W_{aq^{-2(m+(\alpha -1)l)},bq^{-(m+(\alpha-1)l)};q,p}(m-2)\cdot 
 W_{aq^{-2(2+(\alpha -1)(l-k))},bq^{-(2+(\alpha-1)(l-k))};q,p}(1)\\
 &= W_{aq^{-2(m+(\alpha -1)l)},bq^{-(m+(\alpha-1)l)};q,p}(m+k(\alpha -1)-1).
\end{align*}
Adding this weight to the weighted sum of the rook placements so far,
which is
$$[m+k(\alpha -1)-1]_{aq^{-2(m+(\alpha -1)l)},bq^{-(m+(\alpha-1)l)};q,p},$$
gives 
$$[m+k(\alpha -1)]_{aq^{-2(m+(\alpha -1)l)},bq^{-(m+(\alpha-1)l)};q,p},$$
which is the coefficient of $r_k ^{(\alpha)}(a,b;q,p;B)$.
\end{proof}

We can now prove an elliptic analogue of 
the $\alpha$-factorization theorem.

\begin{theorem}\label{thm:ellalpha}
For any Ferrers board $B=B(b_1,b_2,\dots, b_n)$, we have 
\begin{multline}\label{eqn:ellalphaprod}
\qquad\prod_{j=1}^n [z+b_j+(j-1)
(\alpha -1)]_{aq^{-2(b_j+(j-1)(\alpha-1))},bq^{-(b_j+(j-1)(\alpha-1))};q,p}\\
=\sum_{k=0}^n r_{n-k}^{(\alpha)} (a,b;q,p;B)
\prod_{i=1}^k[z+(i-1)
(\alpha -1)]_{aq^{-2(i-1)(\alpha-1)},bq^{-(i-1)(\alpha-1)};q,p}. \qquad
\end{multline} 
\end{theorem}

\begin{proof}
Let us extend the board by attaching $z$ rows of width $n$ below the
board $B$, denoted by $B_z$, and compute 
$$\sum_{P\in\mathcal F_n (B_z)}wt_\alpha (P)$$
in two different ways. The left-hand side of \eqref{eqn:ellalphaprod}
is the result of computing the above weight sum columnwise, and the
right-hand side can be obtained by computing the weight of the cells
in $B$ and and the cells in the extended part separately. 

If we place rooks columnwise starting from the first column, then by
using a similar argument as in Proposition~\ref{prop:elptrecur},
it is not hard to see that all the possible rook placements
in the $j$-th column contribute to the weight sum 
$[z+b_j+(j-1)(\alpha -1)]_{aq^{-2(b_j+(j-1)(\alpha-1))},bq^{-(b_j+(j-1)(\alpha-1))};q,p}$
and the entire weight sum is the product of all those factors. 

On the other hand, consider the way that we choose a placement
$P\in \mathcal F_{n-k}(B)$ and 
extend it to an $n$-rook placement in $F_{n}(B_z)$ by placing $k$ rooks 
in $B_z -B$. Then the weighted sum of all such rook placements is 
\begin{align*}
&\sum_{\substack{P'\in\mathcal F_n (B_z)\\ P' \cap B=P}}wt_\alpha (P')\\
&=\sum_{\substack{P'\in \mathcal F_n (B_z)\\ P'\cap B=P}}wt_\alpha (P)\cdot
wt_\alpha (P'\cap (B_z -B))\\
&= wt_\alpha (P) \prod_{i=1}^k[z+(i-1)
(\alpha -1)]_{aq^{-2(i-1)(\alpha-1)},bq^{-(i-1)(\alpha-1)};q,p},
\end{align*}
where the last line comes from the weight computation of placing
$k$ rooks columnwise.
We sum the above weight over all the file placements in $F_{n-k}(B)$:
\begin{align*}
& \sum_{P\in \mathcal F_{n-k}(B)}\left(wt_\alpha (P)
[z+(i-1)(\alpha -1)]_{aq^{-2(i-1)(\alpha-1)},bq^{-(i-1)(\alpha-1)};q,p} \right)\\
&=\left( \sum_{P\in \mathcal F_{n-k}(B)} wt_\alpha (P)  \right)
[z+(i-1)(\alpha -1)]_{aq^{-2(i-1)(\alpha-1)},bq^{-(i-1)(\alpha-1)};q,p}\\
&= r_{n-k}^{(\alpha)}(a,b;q,p;B)
[z+(i-1)(\alpha -1)]_{aq^{-2(i-1)(\alpha-1)},bq^{-(i-1)(\alpha-1)};q,p}.
\end{align*}
The right-hand side of \eqref{eqn:ellalphaprod} is the result of
summing the above weights over all $k$,
for $0\le k \le n$. This completes the proof.
\end{proof}

\begin{remark}
\begin{enumerate}
 \item If we take the limit $p\to 0$, $a\to 0$, and then $b\to 0$
(or, $p\to 0$, $b\to 0$, then $a\to \infty$), then $r_k ^{(\alpha)}(0,0;q,0;B)$
becomes $R_k ^{(\alpha)}(B)$ which is a $q$-analogue of $r_k ^{(\alpha)}(B)$
defined by Goldman and Haglund \cite[Section 7]{GH}. Thus,
Theorem \ref{thm:ellalpha} gives a $q$-analogue of the
$\alpha$-factorization theorem:
 \begin{align*}
&\prod_{j=1}^n [z+b_j +(j-1)(\alpha -1)]_q\notag\\
& =\sum_{k=0}^n R_{k}^{(\alpha)}(B) [z]_q [z+(\alpha -1)]_q\cdots
[z+(n-k-1)(\alpha -1)]_q.
 \end{align*}
\item
In \cite{SY1}, we obtained the elliptic analogue of the
$\alpha$-factorization theorem \eqref{eqn:ellalphaprod} by a different
approach. There we introduced an elliptic analogue of a generalized
rook model of Miceli and Remmel \cite{McR} utilizing augmented rook boards.
The alpha-parameter model can be obatined from it
by specializing some parameters. 
For full details, see \cite{SY1}.
\end{enumerate}
 
\end{remark}


\subsection{The $\alpha=1$ case}


Note that the number of file placements $f_k (B):=|\mathcal F_k (B)|$
is called a \emph{file number}. An elliptic analogue of the file number
has been defined in \cite{SY0} by assigning weights $w_{a,b;q,p}(1-j)$
to the cell in $(i,j)$ which are neither below nor contain any rook in 
$P\in \mathcal F_k(B)$. 
In the definition of $wt_\alpha (c)$ in \eqref{def:cellwt}, if we 
set $\alpha=1$, then 
\begin{equation*}
wt_{\alpha=1} (c)=
\begin{cases}
 1, &\text{ if there is a rook above or in $c$,}\\
 W_{aq^{-2j},bq^{-j};q,p}(1), &\text{ otherwise.}
\end{cases}
\end{equation*}
Since $W_{aq^{-2j},bq^{-j};q,p}(1)=w_{aq^{-2j},bq^{-j};q,p}(1)=w_{a,b;q,p}(1-j)$,
$r_k ^{(1)}(a,b;q,p;B)$ coincides with the elliptic analogue of the file number 
$f_k(a,b;q,p;B)$, defined in \cite[Section 5]{SY0}. In fact, $f_k(a,b;q,p;B)$ 
is defined for any skyline board. 


\begin{example}[Elliptic Stirling number of the first kind]
Consider the staircase shape board $St_n =B(0,1,2,\dots, n-1)$. The file 
placement of $n-k$ rooks in $St_n$ counts the number of permutations 
of $\{ 1,2,\dots , n\}$ with $k$ cycles, or the signless Stirling 
numbers of the first kind, denoted by $c(n,k)$. Hence,
$r_{n-k}^{(1)}(a,b;q,p;St_n)$
can be defined as an elliptic analogue of $c(n,k)$.  
Let us use the notation $c_{a,b;q,p}(n,k):=r_{n-k}^{(1)}(a,b;q,p;St_n)$. 
The recurrence relation in Proposition \ref{prop:elptrecur} gives 
a recurrence relation for $c_{a,b;q,p}(n,k)$:
\begin{equation}\label{exp:stirling1}
c_{a,b;q,p}(n+1,k)=[n]_{aq^{-2n},bq^{-n}}c_{a,b;q,p}(n,k)+
W_{aq^{-2n},bq^{-n}}(n)c_{a,b;q,p}(n,k-1),
\end{equation}
with the initial conditions $c_{a,b;q,p}(0,0)=1$ and
$c_{a,b;q,p}(n,k)=0$ for $k<0$ or $k>n$.

Furthermore, if we consider the truncated staircase board
$St_n ^{(r)}=B(b_1,\dots, b_n)$ with 
$b_i=0$ for $i=1,\dots, r$ and $b_i=i-1$ for $i=r+1,\dots, n$,
then $f_{n-k}(St_n ^{(r)})$
equals the number of permutations with $k$ cycles such that
the first $r$ numbers $1,2,\dots, r$
 are in distinct cycles; these are called the
\emph{$r$-restricted Stirling number of the first kind}. 
One can now define
 $c_{a,b;q,p}^{(r)}(n,k):=r_{n-k}^{(1)}(a,b;q,p;St_n ^{(r)})$
as an elliptic analogue of the $r$-restricted
Stirling numbers of 
 the first kind. They satisfy the same recursion \eqref{exp:stirling1}
but with the initial conditions 
 $c_{a,b;q,p}^{(r)}(r-1,r-1)=1$ and $c_{a,b;q,p}^{(r)}(n,k)=0$ for $k<r-1$ or $k>n$.
 For details of the correspondence between file placements and permutations 
 with certain number of cycles, see \cite[Subsections~ 5.1 and 5.2]{SY0}.
\end{example}


\begin{example}[Elliptic analogue of Abel polynomials]
The polynomials $z(z+\mathfrak a n)^{n-1}$ are called
\emph{Abel polynomials} and 
the file numbers $f_{n-k}(A_n)$ for the board
$A_n:=B(0,\mathfrak a n,\dots,\mathfrak a n )$
are its coefficients at the monomials $z^k$, that is, 
$$z(z+\mathfrak a n)^{n-1}=\sum_{k=0}^n f_{n-k}(A_n) z^k.$$
The coefficient $f_{n-k}(A_n)$
counts the number of forests on $n$ labeled vertices composed 
of $k$ rooted trees where each of the vertices can be colored by one
of $\mathfrak a$ colors
and all the $k$ roots must have the first color. If we apply
Theorem~\ref{thm:ellalpha}
for $B=A_n$, then the left-hand side of \eqref{eqn:ellalphaprod} becomes 
an elliptic extension of the Abel polynomial. More precisely, we get 
$$ [z]_{a,b;q,p}([z+\mathfrak a n]_{aq^{-2\mathfrak a n},b q^{-2\mathfrak a n};q,p})^{n-1}
=\sum_{k=0}^n r_{n-k}^{(1)} (a,b;q,p;A_n) ([z]_{a,b;q,p})^k.$$
In particular, $r_{n-k}^{(1)} (a,b;q,p;A_n)$ has a nice closed form expression
$$r_{n-k}^{(1)} (a,b;q,p;A_n) =
\binom{n-1}{k-1}(W_{aq^{-2\mathfrak a n}, bq^{-\mathfrak a n};q,p}(\mathfrak a n))^{k-1}
([\mathfrak a n]_{aq^{-2\mathfrak a n}, bq^{-\mathfrak a n};q,p})^{n-k}$$
which can easily be proved by the recurrence relation \eqref{eqn:elptrecur}.
For a detailed description of the combinatorial interpretation 
for $f_{n-k}(A_n)$ and more general cases, see \cite[Section~5.3]{SY0}.
\end{example}


\subsection{The $\alpha=2$ case} 


In \cite{GH}, Goldman and Haglund observe that 
$$R_k ^{(2)}(St_n)=q^{\binom{n-k}{2}}\begin{bmatrix}n+k-1\\2k\end{bmatrix}_q
\prod_{j=1}^k [2j-1]_q.$$
Unfortunately, the elliptic analogue $r_k ^{(2)}(a,b;q,p;St_n)$ does not 
factor nicely. However, if we take the limit $p\to 0$ and $b\to 0$, 
then the corresponding $a,q$-analogue
$r_k ^{(2)}(a,q;St_n):=r_k ^{(2)}(a,0;q,0;St_n)$
has a closed form expression 
\begin{equation*}\label{eqn:alpha2rk}
r_k ^{(2)}(a,q;St_n)
=q^{-\binom{n+k}{2}+k(k+2)}\begin{bmatrix}n+k-1\\2k\end{bmatrix}_q
\prod_{j=1}^k [2j-1]_q
\frac{(a q;q^{-2})_{n-k} (a q^{1-2n};q^2)_k}{(a q;q^{-4})_n}.
\end{equation*}
Substituting this expression into the $\alpha$-factorization theorem gives 
the identity 
$$
\prod_{j=1}^n [z+2(j-1)]_{aq^{-4(j-1)};q}
=\sum_{k=0}^n r_{k}^{(2)} (a,q;St_n)
\prod_{i=1}^{n-k}[z+i-1]_{aq^{-2(i-1)};q},
$$
where we used the simplified notation $[z]_{a;q}:=[z]_{a,0;q,0}$.
If we express this identity in basic hypergeometric notation,
then it reduces to the following terminating ${}_4\phi_3$ sum
(which is equivalent to the terminating $q$-analogue of Whipple's
${}_3 F_2$ sum listed in \cite[(II.19)]{GRhyp}):
\begin{equation*}\label{eqn:aqalphaprop2}
\frac{(q^{z+2},aq^{z-2n};q^2)_n}{(q^{z+1},aq^{z-n};q)_n} 
=\sum_{k=0}^n \frac{(q^{-n}, q^{n+1}, a^{1/2}q^{-n-1/2}, -a^{1/2}q^{-n-1/2};q)_k}
{(q,-q, q^{-z-n}, aq^{z-n};q)_k}q^k.
\end{equation*}

\begin{remark}
In \cite{GH}, Goldman and Haglund give a bijective proof of 
$$r_k ^{(2)}(St_n)=m_k(K_{n+k-1}),$$
where $m_k(K_n)$ is the number of $k$-edge matchings in the
complete graph $K_n$ of $n$ vertices. Since 
$m_k (K_n)=\binom{n}{2k}\frac{(2k)!}{k!2^k}$, 
the bijection gives 
$$r_k ^{(2)}(St_n)=\binom{n+k-1}{2k}\frac{(2k)!}{k!2^k}.$$
\end{remark}


\section{Rook theory for matchings}\label{sec:match}


Haglund and Remmel \cite{HR} studied the rook theory for matchings
for which they replace permutations by
perfect matchings. Rather than $[n]\times [n]$ (the relevant board
for considering permutations of $n$ numbers), they consider the
following shifted board $B_{2n}$ (not to be confused with $B_z$,
considered earlier) pictured in Figure~\ref{fig:B2n}.

\setlength{\unitlength}{1.2pt}
\begin{figure}[ht]
\begin{picture}(120,116)(0,0)
\multiput(105,0)(0,15){1}{\line(1,0){15}}
\multiput(90,15)(0,15){1}{\line(1,0){30}}
\multiput(75,30)(0,15){1}{\line(1,0){45}}
\multiput(60,45)(0,15){1}{\line(1,0){60}}
\multiput(45,60)(0,15){1}{\line(1,0){75}}
\multiput(30,75)(0,15){1}{\line(1,0){90}}
\multiput(15,90)(0,15){1}{\line(1,0){105}}
\multiput(15,105)(0,15){1}{\line(1,0){105}}
\multiput(15,90)(0,15){1}{\line(0,1){15}}
\multiput(30,75)(0,15){1}{\line(0,1){30}}
\multiput(45,60)(0,15){1}{\line(0,1){45}}
\multiput(60,45)(0,15){1}{\line(0,1){60}}
\multiput(75,30)(0,15){1}{\line(0,1){75}}
\multiput(90,15)(0,15){1}{\line(0,1){90}}
\multiput(105,0)(0,15){1}{\line(0,1){105}}
\multiput(120,0)(0,15){1}{\line(0,1){105}}
\put(20,107){$2$}
\put(35,107){$3$}
\put(50,109){$\cdot$}
\put(64,109){$\cdot$}
\put(78,109){$\cdot$}
\put(88,107){$2n$-$1$}
\put(109,107){$2n$}
\put(6, 94){$1$}
\put(6, 79){$2$}
\put(7, 63){$\cdot$}
\put(7, 48){$\cdot$}
\put(7, 33){$\cdot$}
\put(1,17){$2n$-$2$}
\put(1, 2){$2n$-$1$}
\end{picture}
\caption{$B_{2n}$.}\label{fig:B2n}
\end{figure}
Note that any rook placement $P$ in $[n]\times [n]$ is a partial
permutation which can be extended to a placement $P_\sigma$ of $n$ rooks
corresponding to some permutation $\sigma\in S_n$, where $S_n$ is the
set of permutations of $n$ numbers, $1,2,\dots, n$.
For the board $B_{2n}$, we replace permutations by
perfect matchings of the complete graph $K_{2n}$ on vertices $1,2,\dots,2n$.
That is, for each perfect matching $M$ of $K_{2n}$ consisting of
$n$ pairwise vertex disjoint edges in $K_{2n}$, we let 
$$P_{M}=\{ (i,j) ~|~ i<j \text{ and } \{ i,j\} \in M\}$$
where $(i,j)$ denotes the square in row $i$ and column $j$ of $B_{2n}$
according to the labeling of rows and columns pictured in Figure~\ref{fig:B2n}.
We now define a rook placement to be a subset of some $P_M$ for a
perfect matching $M$ of $K_{2n}$. Given a board $B\subseteq B_{2n}$,
we let $\mathcal{M}_k (B)$ denote the set of $k$ element rook placements
in $B$. The analogue of a skyline board in this setting is a board
$B(a_1,a_2,\dots, a_{2n-1})=\{ (i,i+j)~|~ 1\le i\le 2n-1,1\le j\le a_i\}$.
It is called a \emph{shifted Ferrers board} if
$2n-1\ge a_1\ge a_2 \ge \cdots \ge a_{2n-1} \ge 0$ and the nonzero entries
of $a_i$'s are strictly decreasing. A rook in $(i,j)$ with $i<j$ in a
rook placement \emph{cancels} all cells $(i,s)$ in $B_{2n}$ with $i<s<j$ and
all cells $(t,j)$ and $(t,i)$ with $t<i$.
See Figure~\ref{fig:sboardc} for a specific example of a Ferrers board
and the cells being cancelled by a rook on the shifted board $B_8$.\\
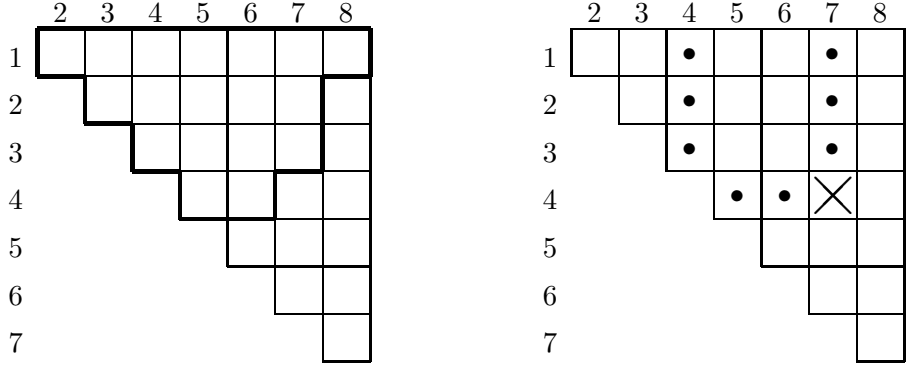
\begin{figure}[ht]
$$\begin{array}{cc}
\begin{picture}(160,123)(0,-3)
\multiput(105,0)(0,15){1}{\line(1,0){15}}
\multiput(90,15)(0,15){1}{\line(1,0){30}}
\multiput(75,30)(0,15){1}{\line(1,0){45}}
\multiput(60,45)(0,15){1}{\line(1,0){60}}
\multiput(45,60)(0,15){1}{\line(1,0){75}}
\multiput(30,75)(0,15){1}{\line(1,0){90}}
\multiput(15,90)(0,15){1}{\line(1,0){105}}
\multiput(15,105)(0,15){1}{\line(1,0){105}}
\multiput(15,90)(0,15){1}{\line(0,1){15}}
\multiput(30,75)(0,15){1}{\line(0,1){30}}
\multiput(45,60)(0,15){1}{\line(0,1){45}}
\multiput(60,45)(0,15){1}{\line(0,1){60}}
\multiput(75,30)(0,15){1}{\line(0,1){75}}
\multiput(90,15)(0,15){1}{\line(0,1){90}}
\multiput(105,0)(0,15){1}{\line(0,1){105}}
\multiput(120,0)(0,15){1}{\line(0,1){105}}
\thicklines\linethickness{1.3pt}
\multiput(15,105)(0,15){1}{\line(1,0){105}}
\multiput(15,90)(0,15){1}{\line(0,1){15}}
\multiput(15,90)(0,15){1}{\line(1,0){15}}
\multiput(30,75)(0,15){1}{\line(0,1){15}}
\multiput(30,75)(0,15){1}{\line(1,0){15}}
\multiput(45,60)(0,15){1}{\line(0,1){15}}
\multiput(45,60)(0,15){1}{\line(1,0){15}}
\multiput(60,45)(0,15){1}{\line(0,1){15}}
\multiput(60,45)(0,15){1}{\line(1,0){30}}
\multiput(90,60)(0,15){1}{\line(1,0){15}}
\multiput(105,90)(0,15){1}{\line(1,0){15}}
\multiput(90,45)(0,15){1}{\line(0,1){15}}
\multiput(105,60)(0,15){1}{\line(0,1){30}}
\multiput(120,90)(0,15){1}{\line(0,1){15}}
\put(20,107){$2$}
\put(35,107){$3$}
\put(50,107){$4$}
\put(65,107){$5$}
\put(80,107){$6$}
\put(95,107){$7$}
\put(110,107){$8$}
\put(6, 93){$1$}
\put(6, 78){$2$}
\put(6, 63){$3$}
\put(6, 48){$4$}
\put(6, 33){$5$}
\put(6, 18){$6$}
\put(6, 3){$7$}
\end{picture}
&
\begin{picture}(120,123)(0,-3)
\multiput(105,0)(0,15){1}{\line(1,0){15}}
\multiput(90,15)(0,15){1}{\line(1,0){30}}
\multiput(75,30)(0,15){1}{\line(1,0){45}}
\multiput(60,45)(0,15){1}{\line(1,0){60}}
\multiput(45,60)(0,15){1}{\line(1,0){75}}
\multiput(30,75)(0,15){1}{\line(1,0){90}}
\multiput(15,90)(0,15){1}{\line(1,0){105}}
\multiput(15,105)(0,15){1}{\line(1,0){105}}
\multiput(15,90)(0,15){1}{\line(0,1){15}}
\multiput(30,75)(0,15){1}{\line(0,1){30}}
\multiput(45,60)(0,15){1}{\line(0,1){45}}
\multiput(60,45)(0,15){1}{\line(0,1){60}}
\multiput(75,30)(0,15){1}{\line(0,1){75}}
\multiput(90,15)(0,15){1}{\line(0,1){90}}
\multiput(105,0)(0,15){1}{\line(0,1){105}}
\multiput(120,0)(0,15){1}{\line(0,1){105}}
\thicklines
\multiput(92,58)(0,15){1}{\line(1,-1){11}}
\multiput(92,47)(0,15){1}{\line(1,1){11}}
\put(20,107){$2$}
\put(35,107){$3$}
\put(50,107){$4$}
\put(65,107){$5$}
\put(80,107){$6$}
\put(95,107){$7$}
\put(110,107){$8$}
\put(6, 93){$1$}
\put(6, 78){$2$}
\put(6, 63){$3$}
\put(6, 48){$4$}
\put(6, 33){$5$}
\put(6, 18){$6$}
\put(6, 3){$7$}
\put(95,95){$\bullet$}
\put(95,80){$\bullet$}
\put(95,65){$\bullet$}
\put(80,50){$\bullet$}
\put(65,50){$\bullet$}
\put(50,95){$\bullet$}
\put(50,80){$\bullet$}
\put(50,65){$\bullet$}
\end{picture}
\end{array}$$
\caption{The shifted Ferrers board $B=(7,5,4,2,0,0,0)\subseteq B_8$,
\hfill\break
and the cells cancelled by a rook in $(4,7)$ on $B_8$.}\label{fig:sboardc}
\end{figure}
Given a shifted Ferrers board $B=B(a_1,\dots, a_{2n-1})\subseteq B_{2n}$
and a rook placement $P\in \mathcal{M}_k (B)$, we let $u_{B}(P)$ denote
the number of cells in $B$ which are neither in $P$ nor rook-cancelled by a
rook in $P$. Then Haglund and Remmel proved the following product formula.
\begin{theorem}\cite{HR}\label{thm:HR}
For a shifted Ferrers board $B=B(a_1,\dots, a_{2n-1})\subseteq B_{2n}$, define 
$$m_{k}(q;B)=\sum_{P\in\mathcal{M}_k (B)}q^{u_{B}(P)}.$$
Then we have 
\begin{equation*}\label{eqn:HRprod}
\prod_{i=1}^{2n-1}[z+a_{2n-i}-2i+2]_q =
\sum_{k=0}^{n}m_k (q;B)[z]\!\downarrow \downarrow _{2n-1-k}
\end{equation*}
where $[z]_q\!\downarrow\downarrow_k=[z]_q [z-2]_q\cdots [z-2k+2]_q$.
\end{theorem}


Here, we shall consider a more generalized case.
Let $\mathbf l=(l_1,\dots,l_N)$ be a fixed
$N$-dimensional vector of positive integers.
For convenience, define $L_0=0$ and $L_j=\sum_{s=1}^jl_s$,
so that $l_j=L_{j}-L_{j-1}$, for $1\le j\le N$.
Now we extend $B_{2n}$ to an \emph{$\mathbf l$-shifted board}
with $L_N=l_1+\dots+l_N$
columns and $N$ rows as in Figure \ref{fig:B_N^l}, denoted by
$B_N ^{\mathbf l}$. Notice that the row labels
successively increase by the increments $l_N,l_{N-1},\dots,l_1$.
\begin{figure}[ht]
\begin{picture}(330,125)(-10,-10)
\multiput(285,0)(0,15){1}{\line(1,0){30}}
\multiput(240,15)(0,15){1}{\line(1,0){75}}
\multiput(165,30)(0,15){1}{\line(1,0){150}}
\multiput(150,45)(0,15){1}{\line(1,0){165}}
\multiput(120,60)(0,15){1}{\line(1,0){195}}
\multiput(60,75)(0,15){1}{\line(1,0){255}}
\multiput(15,90)(0,15){1}{\line(1,0){300}}
\multiput(15,105)(0,15){1}{\line(1,0){300}}
\multiput(15,90)(15,0){3}{\line(0,1){15}}
\multiput(60,75)(15,0){4}{\line(0,1){30}}
\multiput(120,60)(15,0){3}{\line(0,1){45}}
\multiput(150,45)(15,0){3}{\line(0,1){60}}
\multiput(165,30)(15,0){5}{\line(0,1){75}}
\multiput(240,15)(15,0){3}{\line(0,1){90}}
\multiput(285,0)(15,0){3}{\line(0,1){105}}
\put(5, 94){$1$}
\put(-12, 79){$l_N$+$1$}
\put(-12, 63){$l_N$+$l_{N-1}$+$1$}
\put(1, 48){$\cdot$}
\put(1, 33){$\cdot$}
\put(-12,17){$l_N$+\dots+$l_3$+$1$}
\put(-12, 2){$l_N$+\dots+$l_3$+$l_2$+$1$}
\put(18,88){$\underbrace{\qquad\qquad}$}
\put(36,73){$l_N$}
\put(63,73){$\underbrace{\qquad\qquad\quad\;\;}$}
\put(87,58){$l_{N-1}$}
\put(19,108){$2$}
\put(34,108){$3$}
\put(48,108){$\cdots$}
\put(243,13){$\underbrace{\qquad\quad\;\;\;}$}
\put(258,-2){$l_2$}
\put(287,-2){$\underbrace{\qquad\;\;}$}
\put(296,-17){$l_1$}
\put(295,108){$L_N$+$1$}
\put(130,109){$\cdot$}
\put(170,109){$\cdot$}
\put(210,109){$\cdot$}
\end{picture}
\caption{$B_{N}^{\mathbf l}$.}\label{fig:B_N^l}
\end{figure}
For $N=2n-1$ and $\mathbf l=(1,\dots,1)$, the $\mathbf l$-shifted
board $B_{N}^{\mathbf l}$ reduces to the 
shifted board $B_{2n}$ considered by Haglund and Remmel.
A rook placed in $B_{N}^{\mathbf l}$, say $\textbf{r}\in (i,j)$, $i<j$,
attacks the cells in the same row, the same column, and the cells
in the $i$-th column.
We can interpret a rook placement in $B_N^{\mathbf l}$ in the following way.
We call a labeled graph of at most $L_N+1$ vertices from the set
$\{1,2,\dots,L_N+1\}$
\emph{lazy}\footnote{We have chosen the name ``lazy'' since
in contrast to the complete graph we are (lazily) leaving out many edges.
Also, compared to the staircase with a unit increment in column height
after each step, the lazy graph features a staircase where the unit
increments in column height may be delayed.}
with respect to $\mathbf l=(l_1,\dots,l_N)$
(or, an $\mathbf l$-\emph{lazy graph}, in short)
if it only contains edges $(i,j)$
with $i<j$ when $i$ is of the form $l_N+\dots+l_{N-s+1}+1=L_N-L_{N-s}+1$ for
$s\in \{ 0, 1, \dots, N-1\}$.
Then a $k$-rook placement on $B_N^{\mathbf l}$ is a $k$-matching of
$K_{L_N+1}^{\mathbf l}$, the complete
$\mathbf l$-lazy graph on $L_N+1$ vertices.

Given a board $B\subseteq B_{N}^{\mathbf l}$, we let
$\mathcal{M}_k ^{\mathbf l}(B)$
denote the set of placements of $k$ nonattacking rooks in $B$.
An $\mathbf l$-\emph{shifted skyline board} is the set of cells
$$B(a_1,a_2,\dots, a_N)=\{ (L_N-L_{N-i+1}+1,L_N-L_{N-i+1}+1+j)~|~
1\le i\le N,1\le j\le a_i\}.$$
It is called an $\mathbf l$-\emph{shifted Ferrers board} if
$L_N\ge a_1\ge a_2 \ge \cdots \ge a_N\ge 0$ and the nonzero entries
of $a_i$ satisfy $a_i-a_{i+1}\ge l_{N+1-i}$ for $1\le i\le N-1$.
As in the ordinary shifted case, a rook in $(i,j)$
with $i<j$ in a
rook placement cancels all cells $(i,s)$ in $B_N^{\mathbf l}$ with $i<s<j$ and
all cells $(t,j)$ and $(t,i)$ in $B_N^{\mathbf l}$ with $t<i$.
Given an $\mathbf l$-shifted Ferrers board
$B=B(a_1,\dots, a_N)\subseteq B_{N}^{\mathbf l}$
and a rook placement $P\in \mathcal{M}_k ^{\mathbf l}(B)$,
let $u_B ^{(\mathbf l)}(P)$ denote
the number of cells in $B$ which are neither in $P$ nor cancelled by
any rook in $P$. Define 
\begin{equation*}\label{def:mkl}
m_{k}^{(\mathbf l)}(q;B)=\sum_{P\in\mathcal{M}_{k}^{\mathbf l} (B)}
q^{u_{B}^{(\mathbf l)}(P)}.
\end{equation*}
Then we can prove the following product formula. 

\begin{theorem}\label{thm:HR^l}
For any $\mathbf l$-shifted Ferrers board
$B=B(a_1,\dots, a_{N})\subseteq B_{N}^{\mathbf l}$, we have
\begin{equation}\label{eqn:HRprod^l}
\prod_{i=1}^{N}[z+a_{N-i+1}-2i+2]_q =
\sum_{k=0}^{N}m_{k}^{(\mathbf l)} (q;B)[z]\!\downarrow \downarrow _{N-k}
\end{equation}
where $[z]_q\!\downarrow\downarrow_k=[z]_q [z-2]_q\cdots [z-2k+2]_q$.
\end{theorem}

\begin{proof}
It suffices to prove the theorem for nonnegative integer values of $z$.
Then the result follows by analytic continuation.

We extend the board $B_{N}^{\mathbf l}$ by attaching $z$ many columns
of height $N$ to the right of $B_N^{\mathbf l}$, as pictured in
Figure~\ref{fig:BNzl}, and denote it by $B_{N,z}^{\mathbf l}$.
\begin{figure}[ht]
\begin{picture}(300,135)(0,-15)
\multiput(10,105)(0,15){1}{\line(1,0){285}}
\multiput(10,90)(0,15){1}{\line(1,0){120}}
\multiput(10,90)(15,0){3}{\line(0,1){15}}
\multiput(55,75)(0,15){1}{\line(1,0){75}}
\multiput(55,75)(15,0){4}{\line(0,1){30}}
\multiput(115,60)(0,15){1}{\line(1,0){15}}
\multiput(115,60)(0,15){1}{\line(0,1){45}}
\multiput(175,0)(15,0){5}{\line(0,1){105}}
\multiput(280,0)(15,0){2}{\line(0,1){105}}
\multiput(175,0)(0,15){1}{\line(1,0){120}}
\multiput(160,15)(0,15){6}{\line(1,0){75}}
\multiput(280,15)(0,15){6}{\line(1,0){15}}
\thicklines\linethickness{1.3pt}
\multiput(205,0)(15,0){1}{\line(0,1){105}}
\put(130, 45){$\cdot$}
\put(140,35){$\cdot$}
\put(150,25){$\cdot$}
\put(245, 50){$\cdot$}
\put(255,50){$\cdot$}
\put(265,50){$\cdot$}
\put(13, 88){$\underbrace{\qquad\qquad}$}
\put(30, 73){$l_N$}
\put(58,73){$\underbrace{\qquad\qquad\quad\;\;}$}
\put(82,58){$l_{N-1}$}
\put(177, -2){$\underbrace{\qquad\;\;}$}
\put(187, -17){$l_1$}
\put(1, 94){$1$}
\put(-16, 79){$l_N$+$1$}
\put(-16, 60){$l_N$+$l_{N-1}$+$1$}
\put(-4, 45){$\cdot$}
\put(-4, 30){$\cdot$}
\put(-4, 16){$\cdot$}
\put(-16,1){$l_N$+\dots+$l_3$+$l_2$+$1$}
\put(14, 108){$2$}
\put(29, 108){$3$}
\put(70, 110){$\cdot$}
\put(100,110){$\cdot$}
\put(130,110){$\cdot$}
\put(180, 108){$L_N$+$1$}
\put(206, 108){$L_N$+$2$}
\put(256,109){$\cdots$}
\put(275,108){$L_N$+1+$z$}
\end{picture}
\caption{$B_{N,z}^{\mathbf l}$.}\label{fig:BNzl}
\end{figure}
Now we define the set of cells that a rook in $(i,j)$ \emph{attacks}
in the extended board $B_{N,z}^{\mathbf l}$. If $(i,j)\in B_{N}^{\mathbf l}$,
then it attacks as we explained above, that is, the cells in the same row,
in the same column, and the cells in the $i$-th column. If
$(i,j)\in B_{N,z}^{\mathbf l}-B_{N}^{\mathbf l}$,
then the cells that a rook in $(i,j)$ attacks in a rook placement
$P$ depend on other rooks in $P\cap (B_{N,z}^{\mathbf l}-B_{N}^{\mathbf l})$.
That is, if the rook in $(i,j)$, say $\textbf{r}_1$, is the lowest rook in
$P\cap (B_{N,z}^{\mathbf l}-B_{N}^{\mathbf l})$,
then $\textbf{r}_1$ attacks all cells in
row $i$ and column $j$ other than $(i,j)$ and all cells in column $j-1$
if $L_N+2<j$. If $j=L_N+2$, then $\textbf{r}_1$ attacks all cells in row $i$
and column $j$ other than $(i,j)$ plus all cells in column $L_N+1+z$.
In general, if the rook in $(i,j)$ is the $k$-th lowest rook, say
$\textbf{r}_k$, then $\textbf{r}_k$ attacks all cells in row $i$ and
column $j$ other than $(i,j)$ and all cells in the first column of the
following list of columns, $j-1, j-2, \dots, L_N+2, L_N+1+z, L_N+z,\dots, j+1$,
that contain a square which is not attacked by any of the $k-1$ lower rooks
in $B_{N,z}^{\mathbf l} - B_{N}^{\mathbf l}$. In other words,
if $\textbf{r}$ is in $(i,j)$,
then $\textbf{r}$ attacks all cells in column $j$ and the cells in the
first column $s$ in the extended part to the left of column $j$ which
has a cell that is not attacked by any lower rooks in
$P\cap (B_{N,z}^{\mathbf l} - B_{N}^{\mathbf l})$.
If there is no such column to the left
of column $j$, then start scanning from column $L_N+1+z$ and look for
the right-most column containing a square which is not attacked by
any lower rooks in $P\cap (B_{N,z}^{\mathbf l} - B_{N}^{\mathbf l})$.
Note that the existence
of such a column $s$ is guaranteed if $z\ge 2N$. 

Now let $B$ be an $\mathbf l$-shifted Ferrers board contained
in $B_{N}^{\mathbf l}$
and assume that $z\ge 2N$. Let $\mathcal{N}_N(B_{N,z}^{\mathbf l})$
denote the set of all
placements $P$ of $N$ rooks in $B_{N,z}^{\mathbf l}$
such that no cell which contains
a rook in $P$ is attacked by another rook in $P$ and any rook in
$B_{N}^{\mathbf l}\cap P$ is contained in $B$, namely, rooks are not placed
outside of the $\mathbf l$-shifted Ferrers board $B$ in $B_N ^{\mathbf l}$.
To prove \eqref{eqn:HRprod^l}, we define a rook cancellation for a
rook placement $P$ in $\mathcal{N}_N(B_{N,z}^{\mathbf l})$. If a rook $\textbf{r}$
is in $(i,j)\in B$, then we say $\textbf{r}$ $\mathcal{N}$-cancels
all cells in 
\begin{align*}
\{ (r,j): r<i\}\,\cup\, \{(i,s): i+1\le s <j\}\,\cup\,\{ (t,i): t<i \}&\\
\cup\,\{ (i,u): u>j \text{ and } (i,u)\notin B\}&.
\end{align*}
Then note that the cells from the first three sets in this union
agree with the cells that are cancelled by $\textbf{r}$ in $B$ relative to the
$u_B ^{(\mathbf l)}(P\cap B)$ statistic and the last set in the union
contains the cells
to the right of $\textbf{r}$ in $B_{N,z}^{\mathbf l}$ which are not in $B$.
If $\textbf{r}$ is in $(i,j)\in B_{N,z}^{\mathbf l} - B_{N}^{\mathbf l}$, then let
$\mathcal{A}_{(i,j)}^N$ denote the set of cells attacked by $\textbf{r}$.
The rook $\textbf{r}$ then $\mathcal{N}$-cancels all cells in
$\mathcal{A}_{(i,j)}^N$ that lie in rows $s$ with $s<i$ plus all cells
in row $i$ that are either in  $B_{N}^{\mathbf l} -B$ or to the right of $(i,j)$.
Now we let $u_{\mathcal{N}}^{(\mathbf l)}(P)$ denote the number of squares in
$B_{N,z}^{\mathbf l}-P$ which are not $\mathcal{N}$-cancelled by any rooks in $P$.
Then  \eqref{eqn:HRprod^l} is the result of computing the sum 
$$
\sum_{P\in \mathcal{N}_N (B_{N,z}^{\mathbf l})}
q^{u_{\mathcal{N}}^{(\mathbf l)}(P)}
$$
in two different ways. First, we could place $N$ rooks row by row.
Then starting from the right-most cell in the bottom row of $B$, we move
the rook to the left, and then again start from column $L_N+2$
(which is the left-most
cell of the extended part) and move the rook to the right.
For a graphical explanation, see Figure~\ref{fig:firstrow}.
In the figure, the board $B$ is
outlined by thick lines and the boundary of the extended board is denoted
by double lines. 
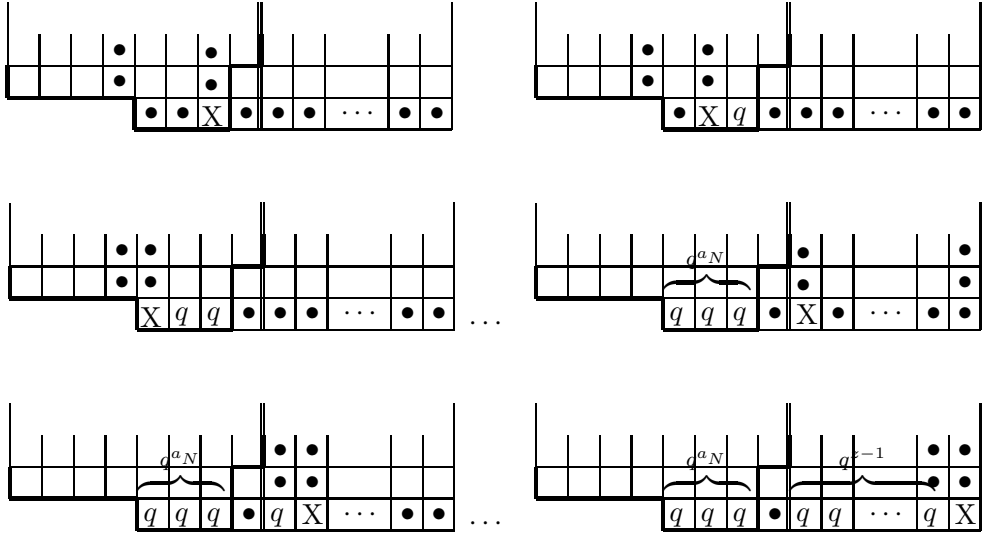
\begin{figure}[ht]
$$\begin{array}{cc}
\begin{picture}(140,40)(0,0)
\multiput(0,10)(10,0){1}{\line(0,1){30}}
\multiput(10,10)(10,0){3}{\line(0,1){20}}
\multiput(0,10)(0,10){1}{\line(1,0){140}}
\multiput(40,0)(10,0){7}{\line(0,1){30}}
\multiput(120,0)(10,0){2}{\line(0,1){30}}
\multiput(140,0)(10,0){1}{\line(0,1){40}}
\multiput(0,20)(0,10){1}{\line(1,0){140}}
\multiput(40,0)(0,10){1}{\line(1,0){100}}
\multiput(79,0)(1,0){2}{\line(0,1){40}}
\thicklines\linethickness{1.3pt}
\multiput(0,10)(10,0){1}{\line(0,1){10}}
\multiput(70,0)(0,10){1}{\line(0,1){20}}
\multiput(0,10)(0,10){1}{\line(1,0){40}}
\multiput(40,0)(10,0){1}{\line(0,1){10}}
\multiput(40,0)(0,10){1}{\line(1,0){30}}
\multiput(70,20)(0,10){1}{\line(1,0){10}}
\multiput(80,20)(0,10){1}{\line(0,1){10}}
\put(61,1){X}
\put(105,3){$\cdots$}
\put(43,3){$\bullet$}
\put(53,3){$\bullet$}
\put(73,3){$\bullet$}
\put(83,3){$\bullet$}
\put(93,3){$\bullet$}
\put(123,3){$\bullet$}
\put(133,3){$\bullet$}
\put(62,12){$\bullet$}
\put(62,22){$\bullet$}
\put(33,13){$\bullet$}
\put(33,23){$\bullet$}
\end{picture}\qquad &
\begin{picture}(140,40)(0,0)
\multiput(0,10)(10,0){1}{\line(0,1){30}}
\multiput(10,10)(10,0){3}{\line(0,1){20}}
\multiput(0,10)(0,10){1}{\line(1,0){140}}
\multiput(40,0)(10,0){7}{\line(0,1){30}}
\multiput(120,0)(10,0){2}{\line(0,1){30}}
\multiput(140,0)(10,0){1}{\line(0,1){40}}
\multiput(0,20)(0,10){1}{\line(1,0){140}}
\multiput(40,0)(0,10){1}{\line(1,0){100}}
\multiput(79,0)(1,0){2}{\line(0,1){40}}
\thicklines\linethickness{1.3pt}
\multiput(0,10)(10,0){1}{\line(0,1){10}}
\multiput(70,0)(0,10){1}{\line(0,1){20}}
\multiput(0,10)(0,10){1}{\line(1,0){40}}
\multiput(40,0)(10,0){1}{\line(0,1){10}}
\multiput(40,0)(0,10){1}{\line(1,0){30}}
\multiput(70,20)(0,10){1}{\line(1,0){10}}
\multiput(80,20)(0,10){1}{\line(0,1){10}}
\put(51,1){X}
\put(105,3){$\cdots$}
\put(43,3){$\bullet$}
\put(62,3){$q$}
\put(73,3){$\bullet$}
\put(83,3){$\bullet$}
\put(93,3){$\bullet$}
\put(123,3){$\bullet$}
\put(133,3){$\bullet$}
\put(52,13){$\bullet$}
\put(52,23){$\bullet$}
\put(33,13){$\bullet$}
\put(33,23){$\bullet$}
\end{picture}\\
\begin{picture}(140,60)(0,0)
\multiput(0,10)(10,0){1}{\line(0,1){30}}
\multiput(10,10)(10,0){3}{\line(0,1){20}}
\multiput(0,10)(0,10){1}{\line(1,0){140}}
\multiput(40,0)(10,0){7}{\line(0,1){30}}
\multiput(120,0)(10,0){2}{\line(0,1){30}}
\multiput(140,0)(10,0){1}{\line(0,1){40}}
\multiput(0,20)(0,10){1}{\line(1,0){140}}
\multiput(40,0)(0,10){1}{\line(1,0){100}}
\multiput(79,0)(1,0){2}{\line(0,1){40}}
\thicklines\linethickness{1.3pt}
\multiput(0,10)(10,0){1}{\line(0,1){10}}
\multiput(70,0)(0,10){1}{\line(0,1){20}}
\multiput(0,10)(0,10){1}{\line(1,0){40}}
\multiput(40,0)(10,0){1}{\line(0,1){10}}
\multiput(40,0)(0,10){1}{\line(1,0){30}}
\multiput(70,20)(0,10){1}{\line(1,0){10}}
\multiput(80,20)(0,10){1}{\line(0,1){10}}
\put(41,1){X}
\put(105,3){$\cdots$}
\put(73,3){$\bullet$}
\put(52,3){$q$}
\put(62,3){$q$}
\put(83,3){$\bullet$}
\put(93,3){$\bullet$}
\put(123,3){$\bullet$}
\put(133,3){$\bullet$}
\put(42,13){$\bullet$}
\put(42,23){$\bullet$}
\put(33,13){$\bullet$}
\put(33,23){$\bullet$}
\end{picture}~\cdots &
\begin{picture}(140,60)(0,0)
\multiput(0,10)(10,0){1}{\line(0,1){30}}
\multiput(10,10)(10,0){3}{\line(0,1){20}}
\multiput(0,10)(0,10){1}{\line(1,0){140}}
\multiput(40,0)(10,0){7}{\line(0,1){30}}
\multiput(120,0)(10,0){2}{\line(0,1){30}}
\multiput(140,0)(10,0){1}{\line(0,1){40}}
\multiput(0,20)(0,10){1}{\line(1,0){140}}
\multiput(40,0)(0,10){1}{\line(1,0){100}}
\multiput(79,0)(1,0){2}{\line(0,1){40}}
\thicklines\linethickness{1.3pt}
\multiput(0,10)(10,0){1}{\line(0,1){10}}
\multiput(70,0)(0,10){1}{\line(0,1){20}}
\multiput(0,10)(0,10){1}{\line(1,0){40}}
\multiput(40,0)(10,0){1}{\line(0,1){10}}
\multiput(40,0)(0,10){1}{\line(1,0){30}}
\multiput(70,20)(0,10){1}{\line(1,0){10}}
\multiput(80,20)(0,10){1}{\line(0,1){10}}
\put(42,3){$q$}
\put(105,3){$\cdots$}
\put(73,3){$\bullet$}
\put(52,3){$q$}
\put(62,3){$q$}
\put(93,3){$\bullet$}
\put(82,2){X}
\put(123,3){$\bullet$}
\put(133,3){$\bullet$}
\put(82,12){$\bullet$}
\put(82,22){$\bullet$}
\put(133,13){$\bullet$}
\put(133,23){$\bullet$}
\put(40, 12){$\overbrace{\qquad\quad}^{q^{a_N}}$}
\end{picture}\\
\begin{picture}(140,60)(0,0)
\multiput(0,10)(10,0){1}{\line(0,1){30}}
\multiput(10,10)(10,0){3}{\line(0,1){20}}
\multiput(0,10)(0,10){1}{\line(1,0){140}}
\multiput(40,0)(10,0){7}{\line(0,1){30}}
\multiput(120,0)(10,0){2}{\line(0,1){30}}
\multiput(140,0)(10,0){1}{\line(0,1){40}}
\multiput(0,20)(0,10){1}{\line(1,0){140}}
\multiput(40,0)(0,10){1}{\line(1,0){100}}
\multiput(79,0)(1,0){2}{\line(0,1){40}}
\thicklines\linethickness{1.3pt}
\multiput(0,10)(10,0){1}{\line(0,1){10}}
\multiput(70,0)(0,10){1}{\line(0,1){20}}
\multiput(0,10)(0,10){1}{\line(1,0){40}}
\multiput(40,0)(10,0){1}{\line(0,1){10}}
\multiput(40,0)(0,10){1}{\line(1,0){30}}
\multiput(70,20)(0,10){1}{\line(1,0){10}}
\multiput(80,20)(0,10){1}{\line(0,1){10}}
\put(42,3){$q$}
\put(105,3){$\cdots$}
\put(73,3){$\bullet$}
\put(52,3){$q$}
\put(62,3){$q$}
\put(92,2){X}
\put(82,3){$q$}
\put(123,3){$\bullet$}
\put(133,3){$\bullet$}
\put(83,13){$\bullet$}
\put(83,23){$\bullet$}
\put(93,13){$\bullet$}
\put(93,23){$\bullet$}
\put(40, 12){$\overbrace{\qquad\quad}^{q^{a_N}}$}
\end{picture}~\cdots &
\begin{picture}(140,60)(0,0)
\multiput(0,10)(10,0){1}{\line(0,1){30}}
\multiput(10,10)(10,0){3}{\line(0,1){20}}
\multiput(0,10)(0,10){1}{\line(1,0){140}}
\multiput(40,0)(10,0){7}{\line(0,1){30}}
\multiput(120,0)(10,0){2}{\line(0,1){30}}
\multiput(140,0)(10,0){1}{\line(0,1){40}}
\multiput(0,20)(0,10){1}{\line(1,0){140}}
\multiput(40,0)(0,10){1}{\line(1,0){100}}
\multiput(79,0)(1,0){2}{\line(0,1){40}}
\thicklines\linethickness{1.3pt}
\multiput(0,10)(10,0){1}{\line(0,1){10}}
\multiput(70,0)(0,10){1}{\line(0,1){20}}
\multiput(0,10)(0,10){1}{\line(1,0){40}}
\multiput(40,0)(10,0){1}{\line(0,1){10}}
\multiput(40,0)(0,10){1}{\line(1,0){30}}
\multiput(70,20)(0,10){1}{\line(1,0){10}}
\multiput(80,20)(0,10){1}{\line(0,1){10}}
\put(42,3){$q$}
\put(105,3){$\cdots$}
\put(73,3){$\bullet$}
\put(52,3){$q$}
\put(62,3){$q$}
\put(92,3){$q$}
\put(82,3){$q$}
\put(122,3){$q$}
\put(132,2){X}
\put(123,13){$\bullet$}
\put(123,23){$\bullet$}
\put(133,13){$\bullet$}
\put(133,23){$\bullet$}
\put(40, 12){$\overbrace{\qquad\quad}^{q^{a_N}}$}
\put(80,12){$\overbrace{\qquad\qquad\quad}^{q^{z-1}}$}
\end{picture}
\end{array}
$$
\caption{Possible rook placements in the bottom first row.}\label{fig:firstrow}
\end{figure}
All the possible rook placements in the bottom row contribute
$1+q+\cdots +q^{z+a_N -1}=[z+a_N]_q$.
Since this rook $\mathcal{N}$-cancels exactly two cells in the above row,
the possible rook placements in the second row from the bottom contribute
$[z+a_{N-1}-2]_q$.  Continuing this way, we get 
\begin{equation*}
\sum_{P\in \mathcal{N}_{N}(B_{N,z}^{\mathbf l})}q^{u_{\mathcal{N}}^{(\mathbf l)}(P)}=
\prod_{i=1}^{N}[z+a_{N-i+1}-2i+2]_q.
\end{equation*}

On the other hand, we could fix a placement
$P\in \mathcal{M}_k ^{\mathbf l}(B)$ and consider the sum 
$$
\sum_{\substack{P'\in \mathcal{N}_{N}(B_{N,z}^l)\\P'\cap B=P}}
q^{u_{\mathcal{N}}^{(\mathbf l)}(P')}.
$$
The way how the $\mathcal{N}$-cancellation is defined ensures that
for any $P'\in\mathcal{N}_N (B_{N,z}^{\mathbf l})$ such that $P' \cap B=P$,
the number of squares of $B_{N}^{\mathbf l}-P$ which are not
$\mathcal{N}$-cancelled
by some rook in $P'$ is consistent with $u_B ^{(\mathbf l)}(P)$.
The same type of argument that we used in the first case applies for
the possible placements of $N-k$ rooks in $B_{N,z}^{\mathbf l}-B_{N}^{\mathbf l}$
which gives
$[z]_q [z-2]_q \cdots [z-2(N-k)+2]_q=[z]_q\!\downarrow\downarrow _{N-k}$.
Thus we have 
\begin{align*}
\sum_{P\in \mathcal{N}_{N}(B_{N,z}^l)}q^{u_{\mathcal{N}}^{(\mathbf l)}(P)}&=
\sum_{k=0}^N \sum_{P\in \mathcal{M}_k ^{\mathbf l}(B)}q^{u_B ^{(\mathbf l)}(P)}[z]_q
\!\downarrow\downarrow _{N-k}\\
&= \sum_{k=0}^N m_k ^{(\mathbf l)}(q;B)[z]_q\!\downarrow\downarrow _{N-k},
\end{align*}
as desired.
\end{proof}

\begin{remark}
Note that when $\mathbf l=(1,\dots,1)$ and $N=2n-1$,
we recover Theorem~\ref{thm:HR} of Haglund and Remmel.  
It is clear from the definition of attacking rooks that in this case 
there can be at most $n$ rooks on the board, so on the right-hand
side of \eqref{eqn:HRprod^l}
the index $k$ can only vary between $0$ and $n$. Hence, the placements 
of $k$ nonattacking rooks on $B_{2n}$ correspond to $k$-matchings of the 
complete graph $K_{2n}$ and $n$-matchings are perfect matchings.
In the general case, the placements of $k$ nonattacking rooks
on $B_{N}^{\mathbf l}$ correspond to $k$-matchings of $K_{L_N+1}^{\mathbf l}$,
while for $L_N+1\ge 2N$ any such placements of $N$
nonattacking rooks correspond to \emph{maximal matchings}.
\end{remark}

The $z=0$ case of Theorem~\ref{thm:HR^l} immediately gives the following
product formula for $m_k^{(\mathbf l)}(q;B)$ when $k=N$.

\begin{corollary}
Given an $\mathbf l$-shifted Ferrers board
$B=B(a_1,\dots, a_N)\subseteq B_N ^{\mathbf l}$, we have
\begin{equation*}
m_N ^{(\mathbf l)}(q;B)=\prod_{i=1}^N[a_{N-i+1}-2i+2]_q.
\end{equation*}
In particular, if $B=B_N ^{\mathbf l}$ (the full $\mathbf l$-shifted board), then 
\begin{equation*}\label{eq:qmaxmatch}
m_N ^{(\mathbf l)}(q;B_N ^{\mathbf l})=\prod_{i=1}^N [L_i-2i+2]_q.
\end{equation*}
\end{corollary}


Now we work out an elliptic analogue of Theorem~\ref{thm:HR^l}.
We essentially assume the same rook cancellation.
However, for the purpose
of conveniently computing the elliptic weights of cells, we
label the rows and columns
of $B_{N}^{\mathbf l}$ as in Figure~\ref{fig:weightedB_N^l}, namely,
we label the columns from $1$ to $L_N$, from right to left,
and label the rows from $1$ to $N$ from the bottom.
\begin{figure}[ht]
\begin{picture}(330,130)(0,-10)
\multiput(285,0)(0,15){1}{\line(1,0){30}}
\multiput(240,15)(0,15){1}{\line(1,0){75}}
\multiput(165,30)(0,15){1}{\line(1,0){150}}
\multiput(150,45)(0,15){1}{\line(1,0){165}}
\multiput(120,60)(0,15){1}{\line(1,0){195}}
\multiput(60,75)(0,15){1}{\line(1,0){255}}
\multiput(15,90)(0,15){1}{\line(1,0){300}}
\multiput(15,105)(0,15){1}{\line(1,0){300}}
\multiput(15,90)(15,0){3}{\line(0,1){15}}
\multiput(60,75)(15,0){4}{\line(0,1){30}}
\multiput(120,60)(15,0){3}{\line(0,1){45}}
\multiput(150,45)(15,0){3}{\line(0,1){60}}
\multiput(165,30)(15,0){5}{\line(0,1){75}}
\multiput(240,15)(15,0){3}{\line(0,1){90}}
\multiput(285,0)(15,0){3}{\line(0,1){105}}
\put(2, 94){$N$}
\put(0, 79){$N$-$1$}
\put(2, 63){$\cdot$}
\put(2, 48){$\cdot$}
\put(2, 33){$\cdot$}
\put(2,17){$2$}
\put(2, 2){$1$}
\put(18,88){$\underbrace{\qquad\qquad}$}
\put(36,73){$l_N$}
\put(63,73){$\underbrace{\qquad\qquad\quad\;\;}$}
\put(87,58){$l_{N-1}$}
\put(15,108){$L_N$}
\put(30,108){$L_N$-$1$}
\put(60,108){$\cdots$}
\put(242,13){$\underbrace{\qquad\qquad}$}
\put(258,-2){$l_2$}
\put(287,-2){$\underbrace{\qquad\;\;}$}
\put(296,-17){$l_1$}
\put(275,108){$3$}
\put(290,108){$2$}
\put(305,108){$1$}
\put(130,109){$\cdot$}
\put(170,109){$\cdot$}
\put(210,109){$\cdot$}
\end{picture}
\caption{Labeling for elliptic weights in $B_{N}^{l}$.}\label{fig:weightedB_N^l}
\end{figure}
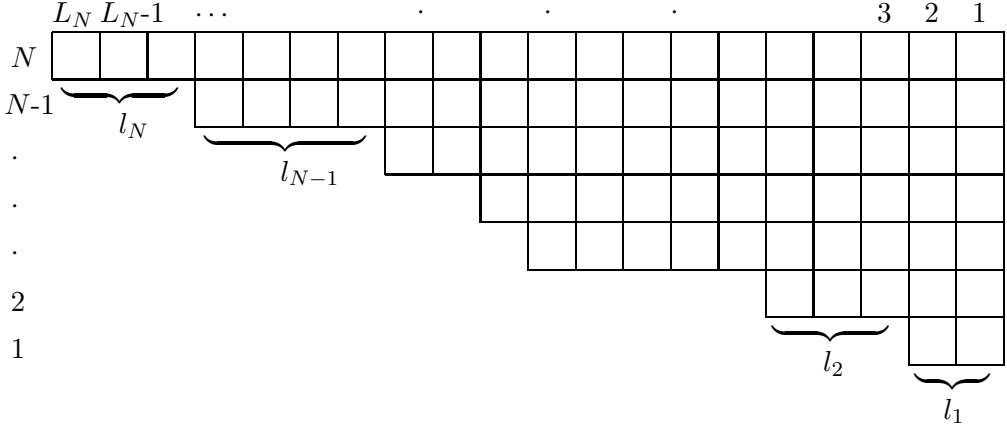
When we use the labeling in Figure \ref{fig:weightedB_N^l}, then
we denote the board by $^w\! B_{N}^{\mathbf l}$ and use $(i,j)^w$ to denote a
cell in row $i$ and column $j$ with respect to this labeling.

Given an $\mathbf l$-shifted Ferrers board
$B=B(a_1,\dots, a_N)\subseteq B_{N}^{\mathbf l}$
and a rook placement $P\in \mathcal{M}_k ^{(\mathbf l)} (B)$,
let $U_{B}^{\mathbf l}(P)$
denote the set of cells in $B - P$ which are not cancelled by
any rook of $P$.
Define
\begin{equation}\label{eqn:mwt}
wt_m(P)=\prod_{(i,j)^w\in U_{B} ^{\mathbf l} (P)}
w_{a,b;q,p}\left(i+j-1-l_i-2r_{(i,j)}(P)-s_{(i,j)}(P)\right),
\end{equation}
where the elliptic weight $w_{a,b;q,p}(l)$ of an integer $l$ is
defined in \eqref{def:smallelpwt}, $r_{(i,j)}(P)$ is the number of rooks
in $P$ positioned south-east
of $(i,j)^w$ such that the two columns cancelled by those rooks are
to the right of the column $j$, and $s_{(i,j)}(P)$ is the number of rooks
in $P$ which are in the south-east region of $(i,j)^w$ such that only
one cancelled column (the column containing the rook)
is to the right of column $j$. Then we define  
\begin{equation*}
m_k ^{(\mathbf l)}(a,b;q,p;B)=\sum_{P\in\mathcal{M}_k ^{\mathbf l}(P)}wt_m(P).
\end{equation*}

\begin{theorem}\label{thmpfm}
For any $\mathbf l$-shifted Ferrers board
$B=B(a_1,\dots, a_N)\subseteq B_{N}^{\mathbf l}$, we have 
\begin{align}\label{eqn:mthm}
&\prod_{i=1}^{N}[z+a_{N-i+1}-2i+2]_{aq^{2(L_{i-1}+i-1-a_{N-i+1})},
bq^{L_{i-1}+i-1-a_{N-i+1}};q,p}
\notag\\
&=\sum_{k=0}^N m_k ^{(\mathbf l)} (a,b;q,p;B)
\prod_{j=1}^{N-k}[z-2j+2]_{aq^{2(L_{j-1}+j-1)},bq^{L_{j-1}+j-1};q,p}.
\end{align}
\end{theorem}

\begin{proof}
It suffices to prove the theorem for nonnegative integer values of $z$.
The full result follows then by analytic continuation. 

The proof is similar to the proof of Theorem~\ref{thm:HR^l}. 
Here we also consider the extended board $B_{N,z}^{\mathbf l}$. However,
we use a different labeling for the sake of elliptic weight computation.
For the cells in $B_{N}^{\mathbf l}$, we use the labeling described in
Figure~\ref{fig:weightedB_N^l}, and for the extended part,
we use the labeling described in Figure \ref{fig:weightedextlb}.
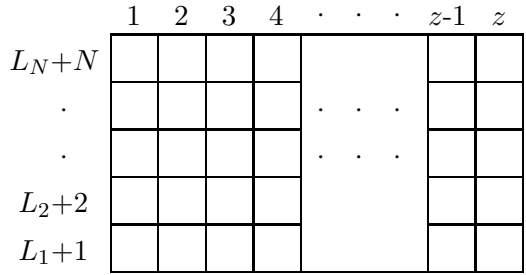
\begin{figure}[ht]
$$\begin{picture}(145,85)(0,0)
\multiput(10,15)(0,15){4}{\line(1,0){60}}
\multiput(10,0)(0,75){2}{\line(1,0){130}}
\multiput(10,0)(15,0){5}{\line(0,1){75}}
\multiput(110,0)(15,0){3}{\line(0,1){75}}
\multiput(110,15)(0,15){4}{\line(1,0){30}}
\put(-22,64){$L_N$+$N$}
\put(-6,34){$\cdot$}
\put(-6,49){$\cdot$}
\put(-19,19){$L_2$+2}
\put(-19,4){$L_1$+1}
\put(15, 78){1}
\put(30,78){2}
\put(45,78){3}
\put(60,78){4}
\put(75,79){$\cdot$}
\put(87,79){$\cdot$}
\put(99,79){$\cdot$}
\put(110,78){$z$-$1$}
\put(130,78){$z$}
\put(75,34){$\cdot$}
\put(87,34){$\cdot$}
\put(99,34){$\cdot$}
\put(75,49){$\cdot$}
\put(87,49){$\cdot$}
\put(99,49){$\cdot$}
\end{picture}$$
\caption{Labeling of the extended part in $B_{N,z}^{\mathbf l}$ for the
elliptic weights.}\label{fig:weightedextlb}
\end{figure}
We consider the rook placements in $\mathcal{N}_N (B_{N,z}^{\mathbf l})$.
We assume the same $\mathcal{N}$-cancellation as defined in the proof
of Theorem~\ref{thm:HR^l}. For a rook placement
$P\in \mathcal{N}_N (B_{N,z}^{\mathbf l})$, let $U_{\mathcal{N}}^{\mathbf l}(P)$
denote the set of cells in $B_{N,z}^{\mathbf l} - P$ which are not
$\mathcal{N}$-cancelled by any rooks in $P$. Consider the cells in
$U_{\mathcal{N}}^{\mathbf l}(P)$. Due to the inconsistency of the column labeling
in $B_{N}^{\mathbf l}$ and the extended part of $B_{N,z}^{\mathbf l}$,
we give slightly
different weights to the cells in $B\cap U_{\mathcal{N}}^{\mathbf l}(P)$ and
$(B_{N,z}^{\mathbf l} - B_{N}^{\mathbf l})\cap U_{\mathcal{N}}^{\mathbf l}(P)$.
That is, to the cells
in $B\cap U_{\mathcal{N}}^{\mathbf l}(P)$, we assign the elliptic weight as
defined in \eqref{eqn:mwt}. To the cells in
$(B_{N,z}^{\mathbf l} - B_{N}^{\mathbf l})\cap U_{\mathcal{N}}^{\mathbf l}(P)$, say to
$(i,j)^w\in (B_{N,z}^{\mathbf l} - B_{N}^{\mathbf l})\cap U_{\mathcal{N}}^{\mathbf l}(P)$
in terms of
the weight labeling described in Figure~\ref{fig:weightedextlb}, we assign
the weight $w_{a,b;q,p}(i+j-1-l_i-2\tilde{r}_{(i,j)}(P)-\tilde{s}_{(i,j)}(P))$,
where $\tilde{r}_{(i,j)}(P)$ is the number of rooks in
$(B_{N,z}^{\mathbf l} - B_{N}^{\mathbf l})\cap P$
which are in the south-west region of
$(i,j)^w$ and both of the columns cancelled by those rooks are to the
left of column $j$, and $\tilde{s}_{(i,j)}(P)$ is the number of rooks
in $(B_{N,z}^{\mathbf l} - B_{N}^{\mathbf l})\cap P$
which are in the south-west region of
$(i,j)^w$ with only one column cancelled by the respective rooks being
to the left of column $j$ (and so the other cancelled column is to the
right of column $j$). Then the product formula \eqref{eqn:mthm}
is the result of computing 
\begin{equation*}
\sum_{P\in\mathcal{N}_{N}(B_{N,z}^{\mathbf l})}\widetilde{wt}_m(P),
\end{equation*}
where 
\begin{align*}
\widetilde{wt}_m(P)=&\prod_{(i,j)^w \in B\cap U_{\mathcal{N}}^{\mathbf l}(P)}
w_{a,b;q,p}(i+j-1-l_i-2r_{(i,j)}(P)-s_{(i,j)}(P))\\
&\times \prod_{(i,j)^w \in (B_{N,z}^{\mathbf l} - B_{N}^{\mathbf l})
\cap U_{\mathcal{N}}^{\mathbf l}(P)}
w_{a,b;q,p}(i+j-1-l_i-2\tilde{r}_{(i,j)}(P)-\tilde{s}_{(i,j)}(P))
\end{align*}
in two different ways. We first place $N$ rooks row by row, starting
from the bottom row. Starting from the right-most cell in the bottom
row of B we move a rook to the left, and then start from the left-most
cell of the extended part and move to the right.
Again we refer to Figure~\ref{fig:firstrow} for a graphical example.
From these possible rook placements, we get the sum of weights
\begin{align*}
 &1+ w_{a,b;q,p}(1-a_N) + w_{a,b;q,p}(1-a_N) w_{a,b;q,p}(2-a_N)+
 \cdots +\prod_{i=1+l_1-a_N}^{l_1+z-1}w_{a,b;q,p}(i-l_1)\\
&=[z+a_N]_{aq^{-2a_N},bq^{-a_N};q,p}.
\end{align*}
We remark that the labeling of the extended part in $B_{N,z}^{\mathbf l}$
was defined so that the above sum continues to add up to an elliptic number.
Note that if the first rook $\textbf{r}_1$ is placed in a cell in
$B_N ^{\mathbf l}$, then it $\mathcal{N}$-cancels exactly two cells in
$B_N ^{\mathbf l}$ in
every row above the bottom row. If $\textbf{r}_1$ is placed in
$B_{N,z}^{\mathbf l} - B_{N}^{\mathbf l}$, then it $\mathcal{N}$-cancels
exactly two cells
in every row in $B_{N,z}^{\mathbf l} - B_{N}^{\mathbf l}$ above the bottom row.
Hence,
after placing a rook in the bottom row, the weight sum over all possible
rook placements in the second row from the bottom is of size
$z+a_{N-1}-2$. However, as it happens in the bottom case, we have to
shift $a$ and $b$ up to the coordinate of the first uncancelled cell
which depends on the difference between the length of the row $i$ and
the size of $a_{N-i+1}$, in general. Also, moving one row up shifts
$a$ by $q^2$ and $b$ by $q$. Thus, the weight sum coming from all
possible rook placements in the $i$-th row from the bottom becomes
$[z+a_{N-i+1}-2i+2]_{aq^{2(L_{i-1}+i-1-a_{N-i+1})},bq^{L_{i-1}+i-1-a_{N-i+1}};q,p}$,
and so finally we obtain 
$$
\sum_{P\in\mathcal{N}_{N} (B_{N,z}^{\mathbf l})}\widetilde{wt}_m(P)
=\prod_{i=1}^{N}
[z+a_{N-i+1}-2i+2]_{aq^{2(L_{i-1}+i-1-a_{N-i+1})},
bq^{L_{i-1}+i-1-a_{N-i+1}};q,p}.
$$

On the other hand, we can fix a placement
$P\in \mathcal{M}_k ^{\mathbf l}(B)$ and
consider the sum
$$\sum_{\substack{P'\in\mathcal{N}_{N}(B_{N,z}^{\mathbf l})\\ P'\cap B=P}}
\widetilde{wt}_m(P').$$
Then by the same reasoning used in the proof of
Theorem~\ref{thm:HR^l}, we obtain
\begin{align*}
\sum_{P\in\mathcal{N}_{N}(B_{N,z}^{\mathbf l})}\widetilde{wt}_m(P)&= 
\sum_{k=0}^N \sum_{P\in \mathcal{M}_k ^{\mathbf l}(B)}wt (P)
\prod_{j=1}^{N-k}[z-2j+2]_{aq^{2(L_{j-1}+j-1)},bq^{L_{j-1}+j-1};q,p}\\
&= \sum_{k=0}^n m_k ^{(\mathbf l)}(a,b;q,p;B)\prod_{j=1}^{N-k}
[z-2j+2]_{aq^{2(L_{j-1}+j-1)},bq^{L_{j-1}+j-1};q,p},
\end{align*}
as desired.
\end{proof}

The first corollary is a consequence of specializing the value $z=0$
in Theorem~\ref{thmpfm}.

\begin{corollary}
Given an $\mathbf l$-shifted Ferrers board
$B=B(a_1,\dots, a_N)\subseteq  B_N ^{\mathbf l}$, we have 
\begin{equation*}
m_N ^{(\mathbf l)}(a,b;q,p;B)=
\prod_{i=1}^{N}[a_{N-i+1}-2i+2]_{aq^{2(L_{i-1}+i-1-a_{N-i+1})},bq^{L_{i-1}+i-1-a_{N-i+1}};q,p}.
\end{equation*}
In particular, if $B=B_N ^{\mathbf l}$, then we have 
\begin{equation*}
m_N ^{(\mathbf l)}(a,b;q,p;B_N ^{\mathbf l})=
\prod_{i=1}^{N}[L_i-2i+2]_{aq^{2(i-1-l_i)},bq^{i-1-l_i};q,p}.
\end{equation*}
\end{corollary}

The next corollary concerns the case $\mathbf l=(1,\dots,1)$
and $N=2n-1$ of Theorem~\ref{thmpfm} which gives an elliptic 
analogue of Theorem \ref{thm:HR}.
For the case $\mathbf l=(1,\dots,1)$,
we use $m_k (a,b;q,p;B)$ to denote $m_k ^{(1,1,\dots, 1)}(a,b;q,p;B)$.
\begin{corollary}\label{cor:shiftedmat}
Given a shifted Ferrers board $B=B(a_1,\dots, a_{2n-1})\subseteq B_{2n}$, 
we have 
\begin{align}\label{eqn:mthm1}
&\prod_{i=1}^{2n-1}[z+a_{2n-i}-2i+2]_{aq^{2(2i-2-a_{2n-i})},bq^{2i-2-a_{2n-i}};q,p}\notag\\
&=\sum_{k=0}^n m_k (a,b;q,p;B)\prod_{j=1}^{2n-1-k}[z-2j+2]_{aq^{4j-4},bq^{2j-2};q,p}.
\end{align}
\end{corollary}

The following result concerns the elliptic enumeration of
(perfect) matchings on $K_{2n}$.
\begin{corollary}\label{cor:pmthm}
Given a shifted Ferrers board $B=B(a_1,\dots, a_{2n-1})\subseteq B_{2n}$,
we have
\begin{equation*}\label{eqn:pmthm}
 m_n(a,b;q,p;B)=\frac{\prod_{i=1}^{2n-1}[a_{2n-i}+2n-2i]_{aq^{2(2i-2-a_{2n-i})},
bq^{2i-2-a_{2n-i}};q,p}}{\prod_{i=1}^{n-1}[2n-2i]_{aq^{4i-4},bq^{2i-2};q,p}}.
\end{equation*}
In particular, for the full shifted Ferrers board
$B_{2n}=B(2n-1,2n-2,\dots,1)$ we have
\begin{align*}
 m_n(a,b;q,p;B_{2n}){}&=\frac{\prod_{i=1}^{2n-1}[2n-i]_{aq^{2i-4},
bq^{i-2};q,p}}{\prod_{i=1}^{n-1}[2n-2i]_{aq^{4i-4},bq^{2i-2};q,p}}\notag\\
&=[2n-1]_{aq^{-2},bq^{-1};q,p}[2n-3]_{aq^2,bq;q,p}\dots[1]_{aq^{4n-6},bq^{2n-3};q,p}.
\end{align*}
\end{corollary}
\begin{proof}
In Corollary~\ref{cor:shiftedmat} we let $z\to 2n-2$.
Since
\begin{equation*}
\prod_{j=1}^{2n-1-k}[2n-2j]_{aq^{4j-4},bq^{2j-2};q,p}=0,
\end{equation*}
for $1\le k<n$,
the right-hand side of \eqref{eqn:mthm1} reduces to a single term only,
corresponding to $k=n$.
Simplification then yields the result.
\end{proof}

Similarly, the case when $\mathbf l=(1,\dots,1)$
and $N=2n$ of Theorem~\ref{thmpfm} gives the following result.
\begin{corollary}\label{cor:shiftedmat1}
Given a shifted Ferrers board $B=B(a_1,\dots, a_{2n})\subseteq B_{2n+1}=
B(2n,2n-1,\dots,1)$, we have
\begin{align}\label{eqn:mthm2}
&\prod_{i=1}^{2n}[z+a_{N-i+1}-2i+2]_{aq^{2(2i-2-a_{N-i+1})},
bq^{2i-2-a_{N-i+1}};q,p}
\notag\\
&=\sum_{k=0}^{2n} m_k(a,b;q,p;B)
\prod_{j=1}^{2n-k}[z-2j+2]_{aq^{4j-4},bq^{2j-2};q,p}.
\end{align}
\end{corollary}

As a special case, we obtain an explicit expression for the
elliptic enumeration of (maximal) matchings on $K_{2n+1}$.
\begin{corollary}
Given a shifted Ferrers board $B=B(a_1,\dots, a_{2n})\subseteq B_{2n+1}$,
we have
\begin{equation*}\label{eqn:pmthm2}
 m_n(a,b;q,p;B)=\frac{\prod_{i=1}^{2n}[a_{2n-i+1}+2n-2i+2]_{aq^{2(2i-2-a_{2n-i+1})},
bq^{2i-2-a_{2n-i+1}};q,p}}{\prod_{i=1}^{n}[2n-2i+2]_{aq^{4i-4},bq^{2i-2};q,p}}.
\end{equation*}
In particular, for the full shifted Ferrers board
$B=B_{2n+1}=B(2n,2n-1,\dots,1)$ we have
\begin{align*}
 m_n(a,b;q,p;B_{2n+1}){}&=\frac{\prod_{i=1}^{2n}[2n-i+2]_{aq^{2i-4},
bq^{i-2};q,p}}{\prod_{i=1}^{n}[2n-2i+2]_{aq^{4i-4},bq^{2i-2};q,p}}\notag\\
&=[2n+1]_{aq^{-2},bq^{-1};q,p}[2n-1]_{aq^2,bq;q,p}\dots[3]_{aq^{4n-6},bq^{2n-3};q,p}.
\end{align*}
\end{corollary}
\begin{proof}
In Corollary~\ref{cor:shiftedmat1} we let $z\to 2n$.
As in the proof of Corollary~\ref{cor:pmthm}
the right-hand side of \eqref{eqn:mthm2} then reduces to a single term only,
corresponding to $k=n$.
Simplification then yields the result.
\end{proof}

When the board is the full $\mathbf l$-shifted Ferrers board
$B=B_N ^{\mathbf l}$, the elliptic matchings number
$m_k ^{(\mathbf l)}(a,b;q,p;B_N ^{\mathbf l})$
satisfies the following recursion which can be proved by considering
whether there is a rook or not in the top row.

\begin{proposition}\label{prop:mkrecur}
 \begin{align*}
 m_k ^{(\mathbf l)}(a,b;q,p;B_N ^{\mathbf l})={}&
[L_N -2k+2]_{aq^{2(N-1-l_N),bq^{N-1-l_N}};q,p}
m_{k-1} ^{(\mathbf l)}(a,b;q,p;B_{N-1} ^{\mathbf l'})\notag\\
& +W_{aq^{2(N-1-l_N),bq^{N-1-l_N}};q,p}(L_N -2k)
m_k ^{(\mathbf l)}(a,b;q,p;B_{N-1} ^{\mathbf l'}),
 \end{align*}
 where $\mathbf l'=(l_1,\dots, l_{N-1})$ and $B_{N-1} ^{\mathbf l'}$
is the board obtained by 
 removing the top row from $B_{N} ^{\mathbf l}$.
\end{proposition}

\begin{remark}
In the limit case $p\to 0$ and $b\to 0$, for the shifted Ferrers board
$B_{2n}=B(2n-1, 2n-2,\dots, 2,1)$,
$m_k(a,q;B_{2n}):=m_k(a,0;q,0;B_{2n})$ has a closed form
\begin{equation*}
 m_k(a,q;B_{2n})=q^{k^2-\binom{2n}{2}}\begin{bmatrix}2n\\2k\end{bmatrix}_q 
 \prod_{j=1}^k [2j-1]_q \frac{(aq^{4n-2k-3};q^2)_{2n-k-1}}{(aq^{-1};q^4)_{2n-k-1}}
\end{equation*}
which can be proved by the recursion in Proposition \ref{prop:mkrecur}.

If we let $a\to \infty$ in $m_k(a,q;B_{2n})$, then we obtain 
$$m_k(q;B_{2n})=q^{\binom{2n-2k}{2}}\begin{bmatrix}2n\\2k\end{bmatrix}_q 
 \prod_{j=1}^k [2j-1]_q, $$
which is a $q$-analogue of the $k$-matching number $\binom{2n}{2k}k!!$
of the complete graph $K_{2n}$.
\end{remark}



\end{document}